\documentclass[12pt]{article}

\usepackage{amssymb}
\usepackage{amsmath}
\usepackage{amsthm}
\usepackage{graphicx}
\usepackage{xcolor}

\PassOptionsToPackage{unicode}{hyperref}

\usepackage{hyperref}
\usepackage{float}
\usepackage{todonotes}

\usepackage[T1]{fontenc}

\allowdisplaybreaks

{\theoremstyle{plain}
 \newtheorem{theorem}{Theorem}
\newtheorem{lemma}[theorem]{Lemma}
\newtheorem{corollary}[theorem]{Corollary}
\newtheorem{proposition}[theorem]{Proposition}
\newtheorem{conjecture}[theorem]{Conjecture}
}
{\theoremstyle{definition}
\newtheorem{definition}[theorem]{Definition}
\newtheorem{example}[theorem]{Example}
\newtheorem{question}[theorem]{Question}
\newtheorem{remark}[theorem]{Remark} 
}

\DeclareMathOperator{\RefComp}{{\textsc{RefComp}}}
\DeclareMathOperator{\FactorEq}{{\textsc{FactorEq}}}
\DeclareMathOperator{\FactorRevEq}{{\textsc{FactorRevEq}}}

\def\andd{\ \wedge\ }
\def\suchthat{ \, : \, }

\newcommand{\infw}[1]{\mathbf{#1}}
\newcommand{\N}{\mathbb{N}}
\newcommand{\Nz}{\mathbb{N}_{0}}
\newcommand{\Z}{\mathbb{Z}}
\newcommand{\eps}{\varepsilon}
\DeclareMathOperator{\rep}{rep}
\DeclareMathOperator{\val}{val}
\DeclareMathOperator{\Pal}{Pal}
\DeclareMathOperator{\refl}{ClassRef}
\DeclareMathOperator{\Unr}{Unr}
\DeclareMathOperator{\Refl}{Ref}
\DeclareMathOperator{\Fac}{Fac}

\newcommand{\seqnum}[1]{\href{https://oeis.org/#1}{\rm \underline{#1}}}

\newenvironment{smallarray}[1]
{\null\,\vcenter\bgroup\scriptsize

\arraycolsep=.13885em
\hbox\bgroup$\array{@{}#1@{}}}
{\endarray$\egroup\egroup\,\null}

\begin{document}


\title{The reflection complexity of sequences over finite alphabets}

\author{ }
 
\maketitle

\begin{center}
 {\textsc{Jean-Paul Allouche}} \ \ \
 {\textsc{John M. Campbell}} \ \ \ 
 {\textsc{Shuo Li}} \ \ \\ 
 {\textsc{Jeffrey Shallit}} \ \ \ 
 {\textsc{Manon Stipulanti}} 

 \ 

\end{center}

\begin{abstract}
 In combinatorics on words, the well-studied factor complexity function $\rho_{\infw{x}}$ of a sequence $\infw{x}$ over a finite 
 alphabet counts, for every nonnegative integer $n$, the number of distinct length-$n$ factors of $\infw{x}$. In this paper, we 
 introduce the \emph{reflection complexity} function $r_{\infw{x}}$ to enumerate the factors occurring in a sequence $\infw{x}$, up 
 to reversing the order of symbols in a word. We prove a number of results about the growth properties of $r_{\infw{x}}$ 
 and its relationship with other complexity functions. We also prove a Morse--Hedlund-type result characterizing eventually periodic 
 sequences in terms of their reflection complexity, and we deduce a characterization of Sturmian sequences. We investigate 
 the reflection complexity of quasi-Sturmian, episturmian, $(s+1)$-dimensional billiard, complementation-symmetric Rote, and rich 
 sequences. Furthermore, we prove that if $\infw{x}$ is $k$-automatic, then $r_{\infw{x}}$ is computably $k$-regular, and we use the 
 software \texttt{Walnut} to evaluate the reflection complexity of some automatic sequences, such as the Thue--Morse sequence. We 
 note that there are still many unanswered questions about this reflection measure. 
\end{abstract}

\noindent {\footnotesize \emph{Keywords:} factor complexity, reflection complexity, reversal, automatic sequence, Sturmian sequence, quasi-Sturmian sequence, episturmian sequence, billiard sequence, Rote sequence, rich sequence, Morse-Hedlund theorem, \texttt{Walnut}.}

\noindent {\footnotesize \emph{MSC:} Primary 68R15; Secondary 11B85.}

\section{Introduction}\label{sec:intro}
 Given  an infinite sequence $\infw{x}$ over a finite alphabet, 
 it is natural to study the combinatorial properties of the
 factors of $\infw{x}$. (The term {\it factor}
 refers to a contiguous block occurring in $\infw{x}$.) For example, writing $\Nz = \{0,1,\ldots \}$ and $\N = \{1,2,\ldots \}$, many authors have studied the \emph{factor complexity function} $\rho_{\infw{x}} \colon \Nz \to \N$,
which maps $n\ge 0$ to the number of distinct factors of 
$\infw{x}$ of length $n$. Note that $\rho_{\infw{x}}(0)=1$, 
since every---finite or infinite---sequence has a unique 
factor of length $0$; namely, the empty word. 
 
 Variations on this definition can be considered as a measure of how ``complicated'' a sequence is. For example, the \emph{abelian complexity function} of $\infw{x}$ counts the number of factors of $\infw{x}$ of a given length, where two factors 
 $u$ and $v$ are considered the same if they have the same length and one is a permutation of the other. 
 Similarly, the \emph{cyclic complexity function} 
 $c_{\infw{x}}$, introduced in 2017 
\cite{CassaigneFiciSciortinoZamboni2017}, is equal to 
 the number of length-$n$ factors of $\infw{x}$, up to equivalence under rotations (cyclic permutations). By analogy, the abelian and cyclic complexity functions
 lead us to introduce, in this paper, a \emph{reflection complexity function} on sequences involving reversals.

 In addition to the factor, abelian, and cyclic complexity functions indicated above, there have been many different complexity functions on sequences that have been previously introduced. 
 In this regard, we highlight the following in alphabetical order: 
 additive complexity~\cite{ArdalBrownJungicSahasrabudhe2012},
 arithmetical complexity~\cite{AvgustinovichFonDerFlaassFrid2003},
 gapped binomial complexity~\cite{RigoStipulantiWhiteland2023}, 
 $k$-abelian complexity~\cite{KarhumakiSaarelaZamboni2013},
 $k$-binomial complexity~\cite{RigoSalimov2015},
 Kolmogorov complexity~\cite{Kolmogorov1963},
 Lempel--Ziv complexity~\cite{LempelZiv1976},
 Lie complexity~\cite{BellShallit2022},
 linear complexity~\cite{Niederreiter2003}, 
 maximal pattern complexity~\cite{KamaeZamboni2002},
 maximum order complexity~\cite{ErdmannMurphy1997}, 
 opacity complexity~\cite{Allouche-Yao},
 open and closed complexity~\cite{ParshinaPostic2020},
 palindrome complexity~\cite{AlloucheBaakeCassaigneDamanik2003}, 
 periodicity complexity~\cite{MignosiRestivo2013},
 privileged complexity~\cite{Peltomaki2013},
 relational factor complexity~\cite{CassaigneKarkiZamboni2008}, 
 (initial) (non-)repe\-ti\-tive complexity~\cite{BugeaudKim2019,Moothathu2012}, 
 span and leftmost complexity~\cite{CassaigneGheeraertRestivoRomanaSciortinoStipulanti2023},
 string attractor profile complexity (implicitly
 defined in~\cite{SchaefferShallit2021}; also see~\cite{CassaigneGheeraertRestivoRomanaSciortinoStipulanti2023}), 
and window 
complexity~\cite{CassaigneKaboreTapsoba2010}.
 Also see the references in the surveys in~\cite{Allouche1994,Ferenczi1999,FerencziKasa1999,FiciPuzynina2023}. 
Our reflection complexity function $r_{\infw{x}}$, defined below, does not seem to have been previously studied, but may be thought of as natural in terms of its relationships with automatic 
 sequences such as the Thue--Morse sequence. To begin with, we require the equivalence relation $\sim_r$ defined below. 
\begin{definition}\label{def:simr}
 Let $m,n$ be nonnegative integers.
 Given a finite word $u= u(1) u(2) \cdots u(m)$, its \emph{reversal} is the word $u^R=u(m) u(m-1) \cdots u(1)$, i.e., $u^R(i) = u(m+1-i)$ for all $i\in\{1,\ldots,m\}$.
 A \emph{palindrome\/} is a word that is equal to its reversal.
 Two finite words $u$ and $v$ are \emph{reflectively equivalent} if $v = u$ or $v = u^R$. 
 We denote this equivalence relation by $u \sim_r v$. 
\end{definition}

\begin{example}
 Over the alphabet $\{ \text{{\tt a}}, \text{{\tt b}}, \ldots, \text{{\tt z}} \}$, the English word {\tt reward} is reflectively equivalent to {\tt drawer}, while {\tt deed}, {\tt kayak}, and {\tt level} are palindromes.
\end{example}

\begin{definition}\label{def:reflection-compl}
 Let $\infw{x}$ be a sequence.
 The \emph{reflection complexity function} $r_{\infw{x}}\colon \Nz \to \N$ of $\infw{x}$ maps every $n\ge 0$ to the number of distinct length-$n$ factors of $\infw{x}$, up to equivalence by $\sim_r$. 
\end{definition}

\begin{example}
 Let 
\begin{equation}\label{eq:TMword}
 \infw{t} 
 = 011010011001011010010110011010011\cdots
\end{equation}
 denote the Thue--Morse sequence, 
 where
 the $n$th term in~\eqref{eq:TMword} 
 for $n\ge 1$ 
 is defined as the number of $1$'s, modulo $2$, 
 in the base-2 expansion of $n - 1$. 
 The initial terms of the integer sequence $(r_{\infw{t}}(n))_{n\ge 0}$ are such that 
\begin{equation}\label{eq:rt}
 (r_{\infw{t}}(n))_{n\ge 0} = 1, 2, 3, 4, 6, 6, 10, 10, 13, 12, 16, 16, 20, 20, 22, \ldots. 
\end{equation} 
 We see that $r_{\infw{t}}(2) = 3$, for example, 
 since there are $3$ length-2 factors of $\infw{t}$, up to reflection complexity, i.e., the factors $00$ 
 and $11$ and one member of the equivalence class $\{ 01, 10 \}$, with respect to $\sim_r$. 
\end{example}

 The integer sequence in~\eqref{eq:rt} was not, prior to this paper, included in the On-Line Encyclopedia of Integer Sequences~\cite{Sloane}, which 
 suggests that our notion of ``reflection complexity'' is new.
 (Now it is present as sequence
 \seqnum{A373700}.)
 Also see the work of Krawchuk 
 and Rampersad in~\cite{KrawchukRampersad2018}, 
 which introduced the notion of 
 \emph{cyclic/reversal complexity} for sequences. 
 The evaluation of reflection complexity functions is closely related to the work of Rampersad and Shallit~\cite{RampersadShallit2005}, 
 who investigated sequences $\infw{x}$ such that all sufficiently long factors $w$ have the property  
that $w^{R}$ is not a factor of $\infw{x}$. 
 Also, the evaluation of reflection complexities for sequences 
 is related to the enumeration of palindromes contained in sequences; see, e.g., Fici and Zamboni~\cite{FiciZamboni2013}. 

 This paper is organized as follows. In Section~\ref{sec:background}, we introduce the notation and definitions needed for the paper. 
 In Section~\ref{sec:General results}, we give general results about reflection complexity. In particular, we investigate its growth properties and relationships 
 with other complexity functions. In Section~\ref{sec:graphs}, we give a graph-theoretic interpretation of reflection-equivalent classes and prove an inequality for reflection complexity. In Sections \ref{sectioneventual}, 
 \ref{sec:Sturmian}, and 
 \ref{sectionrich}, 
 respectively, we 
 investigate reflection 
 complexity for 
 eventually periodic sequences, 
 Sturmian sequences and generalizations, and 
 reversal-closed and rich sequences. 
 Next, in Section~\ref{sec:automatic}, we focus on classical automatic sequences and, with the use of the free software \texttt{Walnut}, we prove that the reflection complexity function for automatic sequences is a
 regular sequence. 
 We also study  reflection complexity for famous automatic sequences such as the Thue--Morse sequence.
 Finally, some further research directions and open questions
 are considered in Section \ref{sec:conclusion}.

\section{Preliminaries}\label{sec:background}

\textbf{Generalities.}
For a general reference on words, we cite~\cite{Lothaire1997}.
An \emph{alphabet} is a finite set of elements called \emph{letters}.
A \emph{word} over an alphabet $A$ is a finite sequence of letters from $A$.
The \emph{length} of a word, denoted between vertical bars, is the number of its letters (counting multiplicities).
The \emph{empty word} is the only $0$-length word, denoted by $\eps$.
For all $n\ge 0$, we let $A^n$ denote the 
set of all length-$n$ words over $A$.
We let $A^*$ denote the set of words over $A$, including the empty word and equipped with the concatenation operation.
In order to distinguish finite words and infinite sequences, we write the latter in bold.
Except for complexity functions, we start indexing words and sequences at $1$, unless otherwise specified.
A \emph{factor} of a word or a sequence is one of its (finite and contiguous) subblocks. A \emph{prefix} (resp., {\emph suffix}) is a starting (resp., ending) factor.
Given a word $w$, its $n$th term is written $w(n)$ for $1\leq n\leq |w|$.
The factor starting at position $n$ and ending at position $m$ with $1\leq m\leq n\leq |w|$ is written $w[m..n]$.
We let $\Fac_w$ denote the set of all factors of $w$ and, for each natural number $n$, we let $\Fac_w(n)$ denote the set of all length-$n$ factors of $w$.
A factor $u$ of a word $w$ over $A$ is \emph{right} (resp., \emph{left}) \emph{special} if $ua$ and $ub$ (resp., $au$ and $bu$) are factors of $w$ for some distinct letters $a,b\in A$.
A sequence $\infw{x}$ is \emph{reversal-closed} if, for every factor $w$ of $\infw{x}$, the word $w^R$ is also a factor of $\infw{x}$.
A sequence $\infw{x}$ is \emph{eventually periodic} if there exist finite words $u,v$, with $v$ nonempty, such that $\infw{x}=uv^\omega$ where $v^\omega=vvv\cdots$ denotes the infinite concatenation of $v$.
A sequence that is not eventually periodic is said to be \emph{aperiodic}.
A sequence is said to be \emph{recurrent} if every factor occurs infinitely many times; it is \emph{uniformly recurrent} if each factor occurs with bounded gaps, i.e., for all factors $w$, there is some length $m = m(w)$ such that $w$ occurs in every length-$m$ block.

\medskip

\noindent \textbf{Morphisms.}
 Let $A$ and $B$ be finite alphabets. 
A \emph{morphism} $f \colon A^* \to B^*$ is a map satisfying $f(uv)=f(u)f(v)$ for all $u,v\in A^*$.
In particular, $f(\eps)=\eps$, and $f$ is entirely determined by the images of the letters in $A$.
For an integer $k\ge 2$, a morphism is \emph{$k$-uniform} if it maps each letter to a length-$k$ word.
A $1$-uniform morphism is called a \emph{coding}.
A sequence $\infw{x}$ is \emph{pure morphic} if it is a fixed point of a morphism, i.e., there exist a morphism $f\colon A^* \to A^*$ and a letter $a \in A$ such that $\infw{x} = f^{\omega}(a)$, where $f^{\omega}(a) = \lim_{n\to \infty} f^n(a)$.
A sequence $\infw{y}$ is \emph{morphic} if there exist a pure morphic sequence $\infw{x}$ and a coding $g\colon A^* \to B^*$ such that $\infw{y} = g(\infw{x})$.
We let $E\colon \{0,1\}^*\to \{0,1\}^*$ be the \emph{exchange morphism} defined by $E(0)=1$ and $E(1)=0$.
We naturally extend $E$ to sequences.

\medskip

\noindent \textbf{Numeration systems.}
Let $U=(U(n))_{n\ge 0}$ be an increasing sequence of integers with $U(0)=1$.
Any integer $n$ can be decomposed in a greedy way as $n=\sum_{i=0}^t c(i)\, U(i)$ with non-negative integer coefficients $c(i)$.
The word $c(t)\cdots c(0)\in\mathbb{N}^*$ is said to be the {\em (greedy) $U$-representation} of $n$. 
By convention, the greedy representation of $0$ is the empty word~$\eps$, and the greedy representation of $n>0$ starts with a non-zero digit.
For $c_t\cdots c_0\in\mathbb{N}^*$, we let $\val_U(c(t)\cdots c(0))$ denote the integer $\sum_{i=0}^t c(i)\, U(i)$.
A sequence $U$ satisfying all the above conditions defines a \emph{positional numeration system}.

\medskip

\noindent \textbf{Automatic and regular sequences.}
For the case of integer base numeration systems, a classical reference on automatic sequences is~\cite{AlloucheShallit2003}, while~\cite{RigoMaes2002,Shallit1988} treat the case of more exotic numeration systems.

Let $U=(U(n))_{n\ge 0}$ be an positional numeration system.
A sequence $\mathbf{x}$ is {\em $U$-automatic} if there exists a deterministic finite automaton with output (DFAO) $\mathcal{A}$ such that, for all $n\ge 0$, the $n$th term $\mathbf{x}(n)$ of $\infw{x}$ is given by the output $\mathcal{A}(\rep_U(n))$ of $\mathcal{A}$.
In particular, if $U$ is the sequence of consecutive powers of an integer $k\ge 2$, then $\mathbf{x}$ is said to be {\em $k$-automatic}.

 It is known that a sequence is $k$-automatic if and only if it is the image, under a coding, of a fixed point 
 of a $k$-uniform morphism~\cite{AlloucheShallit2003}. 

A generalization of automatic sequences to infinite alphabets is the following~\cite{AlloucheShallit2003,RigoMaes2002,Shallit1988}.
Let $U=(U(n))_{n\ge 0}$ be a positional numeration system.
A sequence $\infw{x}$ is \emph{$U$-regular} if there exist a column vector $\lambda$, a row vector $\gamma$ and matrix-valued morphism $\mu$ such that $\infw{x}(n)=\lambda \mu(\rep_U(n)) \gamma$.
Such a system of matrices forms a \emph{linear representation} of $\infw{x}$.
In particular, if $U$ is the sequence of consecutive powers of an integer $k\ge 2$, then $\mathbf{x}$ is said to be {\em $k$-regular}.

Another definition of $k$-regular sequences is the following one~\cite{AlloucheShallit2003}.
Consider a sequence $\infw{x}$ and an integer $k\ge 2$.
The \emph{$k$-kernel} of $\infw{x}$ is the set of subsequences of the form $(\infw{x}(k^e n + r))_{n\ge 0}$ where $r\in\{1,2,\ldots,k^e\}$.
A sequence is $k$-regular if the $\Z$-module generated by its $k$-kernel is finitely generated.
A sequence is then $k$-automatic if and only if its $k$-kernel is finite~\cite{AlloucheShallit2003}.

\medskip

\noindent \textbf{Sturmian sequences.} A sequence 
$\infw{x}$ is \emph{Sturmian} if its factor complexity 
function satisfies $\rho_{\infw{x}}(n) = n+1$ (see, 
e.g.,~\cite{AlloucheShallit2003, Lothaire2002}). Sturmian
sequences have minimal factor complexity among all 
non-eventually periodic sequences, as proved by
Morse and Hedlund~\cite{MorseHedlund1938}. 

\begin{theorem}[{\cite{MorseHedlund1938}}]
\label{thm:morse-hedlund}
Let $\infw{x}$ be a sequence and let $\ell$ be the number of distinct letters occurring in $\infw{x}$.
The following properties are equivalent.
\begin{itemize}
 \item[(a)] The sequence $\infw{x}$ is eventually periodic.
 \item[(b)] We have $\rho_{\infw{x}}(n) = \rho_{\infw{x}}(n+1)$ for some $n\ge 0$. 
 \item[(c)] We have $\rho_{\infw{x}}(n)<n+\ell-1$ for some $n\ge 1$.
 \item[(d)] The factor complexity $\rho_{\infw{x}}$ is bounded.
\end{itemize}
\end{theorem}

\begin{remark}\label{rem:morse-hedlund}
 This theorem implies in particular that $\rho_{\infw{x}}(n)$ is either bounded or it satisfies $\rho_{\infw{x}}(n) \geq n+\ell-1$ for all $n$. 
 Thus the minimal factor complexity among all non-eventually periodic sequences is $\rho_{\infw{x}}(n) = n+1$ for all $n$, i.e., the 
 complexity of Sturmian sequences. Actually there is another ``growth gap'' for $\rho_{\infw{x}}$. Recall that a sequence $\infw{x}$ 
 is called {\em quasi-Sturmian} if there exists a constant $C$ such that, for $n$ large enough, one has $\rho_{\infw{x}}(n) = n + C$ 
 (see \cite{Cassaigne1997}; also see \cite{Chernyatev2008}). 
 It is known that if $\infw{x}$ is neither eventually 
 periodic nor quasi-Sturmian, then 
$\rho_{\infw{x}}(n) - n$ tends to infinity: this result 
is due to Coven~\cite[Lemma~1.3]{Coven1975}; 
also see Cassaigne's proof cited 
in~\cite[Proof of Theorem~3, p.~23]{Allouche2000}. Thus 
\begin{itemize}
 \item[(a)] either $\rho_{\infw{x}}(n)$ is bounded,
 which happens if and only if $\infw{x}$ is eventually
 periodic;
 \item[(b)] or else $\rho_{\infw{x}}(n) = n +C$ for
 some constant $C$ and all $n$ large enough, which
 means that $\infw{x}$ is quasi-Sturmian;
 \item[(c)] or else $\rho_{\infw{x}}(n) - n$ tends to
 infinity.
\end{itemize}
One more point (once explained to the first author by 
Jean Berstel) is that if $\rho_{\infw{x}}(n) = n + 1$ for 
$n$ large enough, then $\rho_{\infw{x}}(n) = n + 1$ for 
all $n$. Namely, let $n_0$ be the least integer $n$ for 
which $\rho_{\infw{x}}(n) = n + 1$, and suppose that 
$n_0 >1$. Hence $\rho_{\infw{x}}(n_0-1) \neq n_0$. The 
sequence $\infw{x}$ cannot be eventually periodic, since
its factor complexity is not bounded. Thus, one has 
$\rho_{\infw{x}}(n) \geq n+1$ for all $n$. Hence, in 
particular, $\rho_{\infw{x}}(n_0 - 1) \geq n_0$. Thus
$\rho_{\infw{x}}(n_0 - 1) > n_0$.
Since $\rho_{\infw{x}}$ is non-decreasing, we have that 
$n_0 < \rho_{\infw{x}}(n_0-1) \leq \rho_{\infw{x}}(n_0) = n_0+1$. This gives $\rho_{\infw{x}}(n_0-1) = n_0+1 = \rho_{\infw{x}}(n_0)$,
which is impossible since $\infw{x}$ is not
eventually periodic. In other words, in the second item
above, if $C=1$, then $\infw{x}$ is Sturmian.
\end{remark}

\section{General results}\label{sec:General results}
 Given a sequence $\infw{x}$, we can decompose its 
 factor complexity function $\rho_{\infw{x}}$ and its 
 reflection complexity function $r_{\infw{x}}$ by using 
 the following functions: for all $n\ge 0$, we let
 
\begin{itemize}
 \item[(a)] 
 $\Unr_{\infw{x}}(n)$ denote the number of ``unreflected'' 
 length-$n$ factors $w$ of $\infw{x}$ such that 
 $w^R$ is not a factor of $\infw{x}$; 
 
\item[(b)] 
 $\Refl_{\infw{x}}(n)$ denote the number of ``reflected'' 
 length-$n$ factors $w$ of $\infw{x}$ such that 
 $w^R$ is also a factor of $\infw{x}$; and

\item[(c)] 
 $\Pal_{\infw{x}}(n)$ denote the number of length-$n$ 
 palindrome factors $w$ of $\infw{x}$ (i.e., the 
 {\em palindrome complexity} function of $\infw{x}$ \cite{AlloucheBaakeCassaigneDamanik2003}). 
\end{itemize}

\noindent In particular, we have 
\begin{align}\label{eq:combinationURP} 
\begin{array}{l}
\rho_{\infw{x}} = \Unr_{\infw{x}} + \Refl_{\infw{x}},\\
r_{\infw{x}} = \Unr_{\infw{x}} + \dfrac{1}{2}(\Refl_{\infw{x}} - 
 \Pal_{\infw{x}}) + 
 \Pal_{\infw{x}}.
\end{array}
\end{align} 

\begin{example}
Let $\infw{f} = 01001010\cdots$ denote the Fibonacci 
sequence, which is the fixed point of $0\mapsto 01, 
1\mapsto 0$.
Its length-$5$ factors are $u_1=01001$, $u_2=10010$, $u_3=00101$, $u_4=01010$, $u_5=10100$, and $u_6=00100$.
Observe that $u_1,u_2,u_3,u_5$ are reflected (second type), and $u_4$ and $u_6$ are palindromes (second and third types).
Furthermore $\infw{f}$ has no unreflected factor, because 
Sturmian words are reversal-closed; see
\cite[Theorem~4, p.\ 77]{DroubayPirillo1999}. 
We therefore obtain $r_{\infw{f}}(6) = 0 + \frac{1}{2}(6-2) + 2 =4$.
\end{example}
 
 The relationships among the complexity functions $\Unr_{\infw{x}}$, $\Refl_{\infw{x}}$, $\rho_{\infw{x}}$, $r_{\infw{x}}$, 
 and $\Pal_{\infw{x}}$ motivates the study of the combinations of these 
functions indicated in Equalities~\eqref{eq:combinationURP}. 
This is illustrated below. 

\begin{lemma}\label{lem:unr-ref-pal}
For a sequence $\infw{x}$ and for all $n\ge 1$, we have 
\[
\rho_{\infw{x}}(n) - r_{\infw{x}}(n) =
\frac{1}{2}(\Refl_{\infw{x}}(n) -
\Pal_{\infw{x}}(n))
\]
and
\[ 
2 r_{\infw{x}}(n) - \rho_{\infw{x}}(n) = 
\Unr_{\infw{x}}(n) + 
\Pal_{\infw{x}}(n).
\]
\end{lemma}

\begin{proof}
Immediate from Equalities~\eqref{eq:combinationURP}.
\end{proof}

This lemma implies the following bounds on the ratio $r/\rho$.

\begin{theorem}\label{thm:r-and-rho}
For a sequence $\infw{x}$ and for all $n\ge 0$, we have 
\[
\frac{1}{2}\rho_{\infw{x}}(n) 
\leq \frac{1}{2} (\rho_{\infw{x}}(n) + \Pal_{\infw{x}}(n))
\leq r_{\infw{x}}(n)
\leq \rho_{\infw{x}}(n).
\]
Furthermore, the equality cases are as follows:

\begin{itemize}

\item[(a)] We have $r_{\infw{x}}(n) = \rho_{\infw{x}}(n)$
if and only if every reflected length-$n$ factor of
$\infw{x}$ is a palindrome.

\item[(b)] We have $r_{\infw{x}}(n) = \frac{1}{2} (\rho_{\infw{x}}(n) + \Pal_{\infw{x}}(n))$
if and only if $\infw{x}$ has no unreflected 
length-$n$ factors. In particular, if the sequence 
$\infw{x}$ is reversal-closed, we have 
$r_{\infw{x}} = 
\frac{1}{2} (\rho_{\infw{x}} + \Pal_{\infw{x}})$. 

\item[(c)] We have $r_{\infw{x}}(n) = \frac{1}{2} \rho_{\infw{x}}(n)$ if and only if $\infw{x}$ has no 
palindrome of length $n$ and each of its length-$n$ 
factors is reflected.
\end{itemize}
\end{theorem}

\begin{proof}
The inequalities and the equality cases are immediate 
consequences of Lemma~\ref{lem:unr-ref-pal}.
\end{proof}

\begin{remark}\label{rem:folklore-reversal-recurrent}
 It is known that if a sequence $\infw{x}$ is reversal-closed, then $\infw{x}$ is recurrent: it suffices to adapt the proof 
 of~\cite[Proposition~1, p.~176]{Dekking80-81}, as indicated in~\cite{BrlekLadouceur2003}. 
 Also note that if a sequence $\infw{x}$ is uniformly 
recurrent and contains infinitely many distinct palindromes, 
then $\infw{x}$ is reversal-closed~\cite[Theorem~3.2]{BalkovaPelantovaStarosta2010}.
\end{remark}

One can say more for uniformly recurrent sequences.
The following dichotomy holds.

\begin{theorem}\label{thm:uniformly-recurrent}
Let $\infw{x}$ be a uniformly recurrent sequence. 
Then either it is reversal-closed, or else it has no
long reflected factors (which implies that $\infw{x}$
has no long palindromes). In other words, 
\begin{itemize}
 \item[(a)] 
 either $\rho_{\infw{x}} = \Refl_{\infw{x}}$, which implies the equality
 $r_{\infw{x}} = \frac{1}{2}(\rho_{\infw{x}} + \Pal_{\infw{x}})$;
 \item[(b)]
 or else there exists $n_0$ such that 
 $\rho_{\infw{x}}(n) = \Unr_{\infw{x}}(n)$ 
 for all $n \geq n_0$, which implies
 $r_{\infw{x}}(n) = \rho_{\infw{x}}(n)$
 for all $n \geq n_0$.
\end{itemize}
\end{theorem}

\begin{proof}
 If $\infw{x}$ is reversal-closed, then $r_{\infw{x}} = \frac{1}{2}(\rho_{\infw{x}} + \Pal_{\infw{x}})$ from Theorem~\ref{thm:r-and-rho}(b) above.
Now suppose that $\infw{x}$ has an unreflected factor
$w$. Since $\infw{x}$ is uniformly recurrent, every sufficiently long
factor of $\infw{x}$ contains $w$ as a
factor, which implies that this long factor itself 
is unreflected.
This exactly says that $\Refl_{\infw{x}}(n) = 0$
for $n$ large enough (and in particular 
$\Pal_{\infw{x}}(n) = 0$ for $n$ large enough). This implies from Equalities~\eqref{eq:combinationURP} that, for $n$ large enough, 
$\rho_{\infw{x}}(n) = \Unr_{\infw{x}}(n)$, and so $r_{\infw{x}}(n) = \rho_{\infw{x}}(n)$.
\end{proof}

 Now we exhibit sequences with particular behaviors of their reflection complexity. 

\begin{example}
It is possible to construct an aperiodic automatic
sequence $\infw{x}$ such that $r_{\infw{x}}(n) = \rho_{\infw{x}}(n)$ and $\Pal_{\infw{x}}(n) > 0$ 
for all $n$.
 An example of such a sequence is given by a fixed 
 point of the morphism $0 \mapsto 01$, $1 \mapsto 23$,
 $2 \mapsto 45$, $3 \mapsto 23$, $4 \mapsto 44$, 
 and $5 \mapsto 44$. 
 This sequence has no reflected factors except 
 palindromes, and there is exactly one palindrome 
 of each length $>1$.
\end{example}

\begin{example}\label{ex:halffactor}
 Consider the sequence $\infw{x}$ on $\{0,1\}$ whose $n$th prefix $x_n$ is given recursively as follows: $x_0=01$ and $x_{n+1}=x_n01x_{n}^R$ for all $n\ge 0$.
 See~\cite[Section~3]{BerstelBoassonCartonFagnot2009} or~\cite[Example~3.1]{BalkovaPelantovaStarosta2010}.
 The sequence $\infw{x}$ is uniformly recurrent, reversal-closed, $2$-automatic, 
 and accepted by a DFAO of 6 states 
 (see, e.g., 
 \cite{AlloucheShallit1998}), and contains only a finite number of palindromes.
 Furthermore, for all sufficiently large $n$, we have $r_{\infw{x}}(n) = 
\dfrac{1}{2} \rho_{\infw{x}}(n)$.
\end{example}

\begin{example}
 It is also possible to construct an aperiodic automatic sequence where the only palindromes are of length 1, but there are reflected 
 factors of each length $> 1$. Let $g_n$ be the prefix of length $2^n - 2$ of $(012)^\omega$. Then an example 
 of an automatic sequence satisfying the desired properties is $$ \infw{x} = 3 \, g_1 \, 4 \ 5 
 \, g_2^R \, 6 \ 3 
 \, g_3 \, 4 \ 5 
 \, g_4^R \, 6 \ 3 
 \, g_5 \, 4 \ 5 \, g_6^R \, 6 
 \cdots, $$
 where~\cite[Theorem~1]{RampersadShallit2005} is required (observe that we intertwine the sequences $(3456)^\omega$ and $g_1 g_2^R g_3 g_4^R \cdots$ to build $\infw{x}$).
\end{example}

\begin{example}
 There is an automatic 
sequence $\infw{x}$ over the alphabet $\{0, 1\}$ 
such that $\text{Ref}_{\infw{x}}(n) = \Omega(n)$
and such that $\text{Unr}_{\infw{x}}(n) = \Omega(n)$. 
Namely, consider the image 
under the coding $0, 1, 2 \mapsto 0$ and 
$3,4\mapsto 1$ of the fixed point, starting with 
$0$, of the morphism $0 \mapsto 01$, $1 \mapsto 23$, 
$2 \mapsto 32$, $3 \mapsto 42$, and $4 \mapsto 43$. 
\end{example}

\begin{example}\label{ex:039982}
We also provide a construction of an 
automatic sequence $\infw{x}$ such that
$r_{\infw{x}}(n+1) < r_{\infw{x}}(n)$ for all 
odd $n \geq 3$. In particular, let $\infw{x}$ 
denote the sequence given by applying the coding 
$a, b, d \mapsto 1$ and $c \mapsto 0$ to the fixed 
point, starting with $a$, of the morphism defined by
$a \mapsto ab$, $b \mapsto cd$, 
$c \mapsto cd$, and $d \mapsto bb$. 
This gives us sequence~\cite[{\tt A039982}]{Sloane} in the OEIS. 
Computing the reflection complexity of $\infw{x}$ (e.g., using
\texttt{Walnut}) gives that
\[
r_{\infw{x}}(n)
=
\begin{cases}
n+1, & \text{for odd $n\geq1$}; \\
n-1, & \text{for even $n \geq 4$}.
\end{cases}
\]
Actually we even have 
that $r_{\infw{x}}(n+1) = r_{\infw{x}}(n) - 1$ for 
all odd $n \geq 3$. See Theorem~\ref{thm:Pal-first-diff-rho}.
\end{example}

With extra hypotheses on a sequence $\infw{x}$, 
we can give more precise results in comparing 
the respective growths of reflection 
and factor complexities. We will need 
Theorem~\ref{thm:Pal-first-diff-rho} below.
Note that Part (b) of this theorem was originally 
stated for uniformly recurrent sequences: see~\cite[Theorem~1.2]{BalaviMasakovaPelantova2007}.
However, its proof only requires 
the sequences to be recurrent (see~\cite[p.~449]{BalkovaPelantovaStarosta2010} and also~\cite[Footnote, p.~493]{BrlekReutenauer2011}). Furthermore we have seen
that a reversal-closed sequence must be recurrent (see
Remark~\ref{rem:folklore-reversal-recurrent}). Thus
we can state the theorem as follows (also see
Theorem~\ref{thm:uniformly-recurrent}).

\begin{theorem}\label{thm:Pal-first-diff-rho}
\leavevmode
\begin{itemize}
\item[(a)]
Let $\infw{x}$ be a uniformly recurrent sequence.
If $\infw{x}$ is not closed under reversal, then
$\Pal(n) = 0$ for $n$ large enough (actually one even 
has $\Refl_{\infw{x}}(n) = 0$ for $n$ large enough).

\item[(b)]
Let $\infw{x}$ be a reversal-closed sequence. For all 
$n\ge 0$, we have
\[
 \Pal_{\infw{x}}(n+1) + \Pal_{\infw{x}}(n) \le \rho_{\infw{x}}(n+1)-\rho_{\infw{x}}(n)+2.
\]
\end{itemize}
\end{theorem}

\begin{remark}
There exist sequences that are uniformly recurrent, 
reversal-closed, and have no long palindromes
(see \cite{BerstelBoassonCartonFagnot2009}; also see
Example~\ref{ex:halffactor} above).
\end{remark}

   We deduce the following results from       Theorem~\ref{thm:Pal-first-diff-rho}.  

\begin{theorem}\label{thm:upper-bound-bmp}
 Let $\infw{x}$ be a reversal-closed sequence. 
 For all $n\ge 0$, we have
 \[
 \frac{1}{2} \rho_{\infw{x}}(n) \leq 
 r_{\infw{x}}(n) \leq 
 \frac{1}{2} \rho_{\infw{x}}(n+1) + 1.
 \]
\end{theorem}

\begin{proof}
 Using the first inequality in Theorem~\ref{thm:r-and-rho}, the statement in Theorem~\ref{thm:r-and-rho}(b),
and Theorem~\ref{thm:Pal-first-diff-rho}(b), we have 
 \[
 \frac{1}{2} \rho_{\infw{x}}(n) \leq
 r_{\infw{x}}(n) = \frac{1}{2}(\rho_{\infw{x}}(n) 
 + \Pal_{\infw{x}}(n)) \leq
 \frac{1}{2} (\rho_{\infw{x}}(n) + 
 \rho_{\infw{x}}(n+1)-\rho_{\infw{x}}(n)+2)
 \]
 for all $n\ge 0$.
 The desired inequalities follow.
\end{proof}

\begin{proposition}
 Let $\infw{x}$ be a reversal-closed sequence.
 Then we have $r_{\infw{x}}(n+1)+r_{\infw{x}}(n) 
 \leq \rho_{\infw{x}}(n+1)+1$ for all $n\ge 0$.
\end{proposition}
\begin{proof}
It is enough to combine Theorems~\ref{thm:Pal-first-diff-rho}(b) and~\ref{thm:r-and-rho}(b).
\end{proof}

On the other hand, we can use a result of
\cite{AlloucheBaakeCassaigneDamanik2003} to 
obtain the following theorem.
\begin{theorem}\label{thm:upper-bound-abcd}
 Let $\infw{x}$ be a non-eventually periodic and reversal-closed sequence.
 For all $n\ge 1$, we have
 \[
 \frac{1}{2} \rho_{\infw{x}}(n) \leq
 r_{\infw{x}}(n) < \frac{1}{2} \rho_{\infw{x}}(n) + \frac{8}{n}
\rho_{\infw{x}}\left( n + \left\lfloor\frac{n}{4}\right\rfloor \right).
 \]
\end{theorem}

\begin{proof}
 Given a non-eventually periodic sequence $\infw{x}$, we have from \cite[Theorem~12]{AlloucheBaakeCassaigneDamanik2003} the 
 inequality
 \[
 \Pal_{\infw{x}}(n) < \frac{16}{n} \rho_{\infw{x}}\left( n + \left\lfloor\frac{n}{4}\right\rfloor \right)
 \]
 for all $n\ge 1$.
 The statement follows from this inequality and Theorem~\ref{thm:r-and-rho}.
\end{proof}

\begin{corollary}\label{cor:growth}
Let $\infw{x}$ be a non-eventually periodic and reversal-closed sequence.
If its factor complexity satisfies
$\rho_{\infw{x}}(n+1) \sim \rho_{\infw{x}}(n)$
or
$\frac{\rho_{\infw{x}}(2n)}{\rho_{\infw{x}}(n)} = o(n)$, then 
\[
r_{\infw{x}}(n) \sim \frac{1}{2}\rho_{\infw{x}}(n)
\] when $n$ tends to infinity.
In particular, this equivalence holds if $\infw{x}$ is non-eventually periodic,
reversal-closed, and morphic.
\end{corollary}
\begin{proof}
Let $\infw{x}$ be a non-eventually periodic and reversal-closed sequence.
If $\rho_{\infw{x}}(n+1) \sim \rho_{\infw{x}}(n)$,
then, from Theorem~\ref{thm:upper-bound-bmp},
we obtain that 
$r_{\infw{x}}(n) \sim \frac{1}{2}\rho_{\infw{x}}(n)$ when $n$ tends to infinity.
Now, if $\frac{\rho_{\infw{x}}(2n)}{\rho_{\infw{x}}(n)} = o(n)$, we obtain, from
Theorem~\ref{thm:upper-bound-abcd}, and using the fact that $\rho_{\infw{x}}$ is non-decreasing, that
\[
\frac{1}{2}\rho_{\infw{x}}(n) \leq r_{\infw{x}}(n) 
< \frac{1}{2}\rho_{\infw{x}}(n) + \frac{8}{n} 
\rho_{\infw{x}}(2n) = \frac{1}{2}\rho_{\infw{x}}(n) +o(\rho_{\infw{x}}(n)),
\]
which is enough.

Now suppose that, in addition, the sequence $\infw{x}$ is morphic. We know that either $\rho_{\infw{x}}(n) = 
\Theta(n^2)$ or $\rho_{\infw{x}}(n) = O(n^{3/2})$
(see~\cite{Devyatov2018,Deviatov2008}).
In the first case, then 
$\frac{\rho_{\infw{x}}(2n)}{\rho_{\infw{x}}(n)}$ is
bounded, and hence $o(n)$.
If $\rho_{\infw{x}}(n) = O(n^{3/2})$, since 
${\infw{x}}$ is not eventually periodic (hence 
$\rho_{\infw{x}}(n) \geq n+1$), there exists a constant
$C > 0$ such that
\[
\frac{\rho_{\infw{x}}(2n)}{\rho_{\infw{x}}(n)} 
\leq C\frac{n^{3/2}}{n+1} = o(n).
\]
This finishes the proof.
\end{proof}

 The upper bound in 
Theorem~\ref{thm:upper-bound-abcd} raises questions about the growth properties of the function 
$r_{\infw{x}}$ more generally, apart from the 
case where the set of factors of $\infw{x}$
satisfies the hypotheses of 
Theorem~\ref{thm:upper-bound-abcd}. This leads 
us to the growth property in
Theorem~\ref{thm:paritygrowth} below. 

\begin{theorem}\label{thm:paritygrowth}
Let $\infw{x}$ be a sequence. Then $r_{\infw{x}} (n) \leq r_{\infw{x}} (n+2)$ for all $n \geq 0$.
\end{theorem}

\begin{proof}
The result is clear for $n = 0$, so assume $n > 0$ in what follows.
Let $c$ be a letter not in the alphabet of $\infw{x}$, and define
${\infw{y}} = c {\infw{x}}$. Then $r_{\infw{y}} (n) = r_{\infw{x}} (n) + 1$
for all $n >0$, since $\infw{y}$ has exactly one additional factor for each length $n \geq 1$; namely, the prefix of length $n$. Thus, it suffices to prove the claim for $\infw{y}$
instead of $\infw{x}$.

With each length-$n$ factor $w$ of $\infw{y}$ associate a set $S_w$ of length-$(n+2)$
factors of $\infw{y}$, as follows: if $w$ is the length-$n$ prefix of $\infw{y}$, then $S_w := \{ w' \}$, where $w'$ is the prefix of length $n+2$ of $\infw{y}$.
We call such a factor {\it exceptional}.
Otherwise, define $S_w := \{ z \in \Fac({\infw{y}}) \suchthat z = awb
 \text{ for some letters $a,b$} \}$. Note that the sets $S_w$, over all length-$n$ factors of $\infw{y}$, are pairwise disjoint, and cover all the length-$(n+2)$ factors of $\infw{y}$.

 For a factor $w$ of $\infw{y}$, define $[w]_1 = 1$ if $w$ is a palindrome, and $0$ otherwise.
Similarly, $[w]_2 = 1$ if $w^R$ is not a factor of $\infw{y}$ and $0$ otherwise. 
Finally, define $[w]_3 = 1$ if $w^R$ is also a factor of $\infw{y}$ but
$w$ is not a palindrome, and
$0$ otherwise. Notice that these three cases are disjoint and subsume all possibilities for factors of $\infw{y}$ (also recall the decomposition at the beginning of the section).
We extend this notation to sets by defining
$[S]_i = \sum_{w \in S} [w]_i$ for $i\in\{1,2,3\}$.
Define $[w]= [w]_1 + [w]_2 + [w]_3/2$ and similarly for
$[S]$.
From Equalities~\eqref{eq:combinationURP}, we know that
$$r_{\infw{y}}(n) = \sum_{\substack{|w|=n \\ w \in \Fac({\infw{y}})}} [w] $$
while 
$$ 
r_{\infw{y}}(n+2) = 
\sum_{\substack{|w|=n \\ w \in \Fac({\infw{y}})}} [S_w].
$$
Therefore, to show the desired inequality
$r_{\infw{y}} (n) \leq r_{\infw{y}}(n+2)$, it suffices to show that
$[w] \leq [S_w]$ for all length-$n$ factors $w$ of $\infw{y}$.

Suppose $w$ is exceptional. Recall that $w$ starts with $c$,
which appears nowhere else in $\infw{y}$. Then $[w]_1 = [w]_3 = 0$, but $[w]_2 = 1$.
And $S_w = \{ wab \}$, so $[S_w]_1 = [S_w]_3 = 0$, but $[S_w]_2 = 1$.
Therefore $[w]\leq [S_w]$.

Now suppose $w$ is not exceptional. There are three cases to consider, noting that when $n = 1$, we have that $[w]_{1} = 1$ and $[w]_{2} = 0$. 

\bigskip 

Case 1: If $[w]_1 = 1$, then $w$ is a palindrome.
Consider a factor $awb \in S_w$. If it is 
a palindrome, then $[awb]_1 = 1$, so $[w]\leq [awb]$. If $awb$
is not a palindrome, then $awb \not= (awb)^R = b w^R a = bwa $.
Thus $a\not=b$. If $bwa$ is not a factor of $\infw{y}$,
then $[awb]_2 = 1$, so $[w] \leq [awb]$. If $bwa$ is a factor
of $\infw{y}$, then $bwa\in S_w$ and $[awb]_3 + [bwa]_3 = 2$, so in all cases $[w] \leq [awb]$.
Thus $[w] \leq [S_w]$.

\bigskip

Case 2: If $[w]_2 = 1$, then $w^R$ is not a factor of $\infw{y}$. Consider
a factor $awb \in S_w$. Then $(awb)^R = b w^R a$, so $(awb)^R$
cannot be a factor of $y$ either. Hence $[awb]_2=1$, $[w]\leq [awb]$, and
hence $[w]\leq [S_w] $.

\bigskip

Case 3: If $[w]_3 = 1$, then $w^R$ is a factor of $\infw{y}$, but $w$ is not a 
palindrome. Consider a factor $awb \in S_w$. If $awb$ is a palindrome,
then $awb = (awb)^R = b w^R a$, so $w^R$ would be a palindrome, a contradiction.
So $awb$ is not a palindrome and $[awb]_1=0$.
If $(awb)^R = b w^R a$ is a factor of $\infw{y}$,
 then $[awb]_3 = 1$, so $[w]\leq [awb] $. 
If $(awb)^R$ is not a factor of $\infw{y}$, then 
$[awb]_2 = 1$, so $[w] <[awb]$. Thus $[w]\leq [S_w]$. 

\medskip

This completes the proof.
\end{proof}

\begin{remark}
Another formulation of Theorem~\ref{thm:paritygrowth}
above is that the sequence 
$(r_{\infw{x}}(n+1) + r_{\infw{x}}(n))_{n \geq 0}$ 
is non-decreasing.
\end{remark}

 With regard to the conjectures listed at the end of this section, we can prove a weaker form of Conjecture 
 \ref{conj:growthconj3} for reversal-closed sequences, and a weaker form of Conjecture \ref{conj:growthconj1} 
 for sequences without 
 long palindromes. Also we can prove that 
Conjecture~\ref{conj:growthconj4} holds for primitive
morphic sequence.

\begin{theorem}\label{thm:weak-conj:growthconj1}
Let $n_0\ge 0$ be an integer and let $\infw{x}$ be a sequence with no palindrome of
length $\geq n_0$. Then $(r_{\infw{x}}(n))_{n \geq 0}$
is eventually non-decreasing: 
$r_{\infw{x}}(n) \leq r_{\infw{x}}(n+1)$ for $n \geq n_0$.
Furthermore, if $r_{\infw{x}}(n+2) = r_{\infw{x}}(n)$ 
for some $n \geq n_0$, then the sequence 
$\infw{x}$ is eventually periodic.
\end{theorem}

\begin{proof}
By combining the assumption and the second equality of Lemma~\ref{lem:unr-ref-pal}, we have that
\begin{equation}
\label{eq:equality big n}
 r_{\infw{x}}(n) = \dfrac{1}{2}(\rho_{\infw{x}}(n) + \Unr_{\infw{x}}(n))
\end{equation}
for $n \geq n_0$. Since both 
$(\rho_{\infw{x}}(n))_{n\ge 0}$ and 
$(\Unr_{\infw{x}}(n))_{n\ge 0}$ are 
non-decreasing, we see that $(r_{\infw{x}}(n))_{n \geq n_0}$ is
non-decreasing, which gives that 
$r_{\infw{x}}(n) \leq r_{\infw{x}}(n+1)$ for $n \geq n_0$.
This shows the first part of the statement. 
For the second part, if we have
$r_{\infw{x}}(n+2) = r_{\infw{x}}(n)$ for some 
$n \geq n_0$, then Equality~\eqref{eq:equality big n} implies that $\rho_{\infw{x}}(n+2) + \Unr_{\infw{x}}(n+2) = \rho_{\infw{x}}(n) + \Unr_{\infw{x}}(n)$. Hence
\begin{equation}\label{eq:eventual equality}
 \rho_{\infw{x}}(n+2) + \Unr_{\infw{x}}(n+2) = \rho_{\infw{x}}(n+1) + \Unr_{\infw{x}}(n+1) =
\rho_{\infw{x}}(n) + \Unr_{\infw{x}}(n). 
\end{equation}
Hence $\rho_{\infw{x}}(n+1) = \rho_{\infw{x}}(n)$, 
which implies that $\infw{x}$ is eventually
periodic from Theorem~\ref{thm:morse-hedlund}.
\end{proof}

 Numerical experiments concerning the growth of the reflection complexity have led us to formulate Conjectures 
 \ref{conj:growthconj1}--\ref{conj:growthconj4} below. Actually, Theorems~\ref{thm:uniformly-recurrent} and 
 \ref{thm:weak-conj:growthconj1} can be used to show that 
 Conjecture \ref{conj:growthconj3}
 holds if $\infw{x}$ is uniformly recurrent (also see 
 Theorem \ref{thm:growthconj2} below). 
 In the same vein, Theorem \ref{thm:uniformly-recurrent}
 and Corollary~\ref{cor:growth} imply the following corollary.

\begin{corollary}\label{cor:weak-growthconj4}
Conjecture~\ref{conj:growthconj4} holds for 
non-eventually periodic primitive morphic sequences.
\end{corollary}

\begin{proof}
 Let $\infw{x}$ be a primitive morphic sequence. We know that $\infw{x}$ is uniformly recurrent. Thus, from 
 Theorem~\ref{thm:uniformly-recurrent}, either $\infw{x}$ is reversal-closed, or else it has 
 no long palindromes. If $\infw{x}$ is reversal-closed,
then by Corollary~\ref{cor:growth}, we have 
$r_{\infw{x}}(n) \sim \frac{1}{2} \rho_{\infw{x}}(n)$.
Otherwise, $\infw{x}$ has no long palindromes, then,
still from Theorem~\ref{thm:uniformly-recurrent}, we
have that $r_{\infw{x}}(n) = \rho_{\infw{x}}(n)$
for $n$ large enough.
\end{proof}

 We conclude this section by leaving the following conjectures as open problems. 

\begin{conjecture}\label{conj:growthconj1}
Let $\infw{x}$ be a sequence. Then
$r_{\infw{x}}(n) = r_{\infw{x}}(n+2)$ for some $n$ 
if and only if $\infw{x}$ is eventually periodic. 
\end{conjecture}
Note that one direction is true. We have even more: 
namely, if the sequence $\infw{x}$ is eventually 
periodic, then $r_{\infw{x}}(n) = r_{\infw{x}}(n+2)$, 
for all $n$ large enough 
(see Theorem~\ref{thm:periodicity} below).

\begin{conjecture}\label{conj:growthconj3}
Let $\infw{x}$ be a sequence 
of at most linear factor complexity. 
Then $r_{\infw{x}}(n+1)-r_{\infw{x}}(n)$ is bounded for all $n\ge 0$. 
Hence, in particular, if $\infw{x}$
is (generalized) automatic, 
so is $(r_{\infw{x}}(n+1)-r_{\infw{x}}(n))_{n\ge 0}$.
\end{conjecture}

It can be shown that
Conjecture \ref{conj:growthconj3}
holds for the Thue--Morse, period-doubling, 
Golay--Shapiro, second-bit, paperfolding, 
Stewart choral, Baum-Sweet, Chacon, and Mephisto-Waltz 
sequences. 

\begin{conjecture}\label{conj:growthconj4}
Let $\infw{x}$ be a sequence. If the limit
$\lim_{n \to \infty} \frac{ r_{\infw{x}}(n)}{\rho_{\infw{x}}(n)}$
exists, then it is either equal to $\frac{1}{2}$ 
or to $1$. 
\end{conjecture}

\section{Flye Sainte-Marie--Rauzy graphs for reflec-tion-equivalent classes}\label{sec:graphs}

In this section, we prove the following inequality:
\begin{theorem}\label{thm:growthconj2}
Let $\infw{x}$ be a sequence. Then $r_{\infw{x}} (n) \leq r_{\infw{x}}(n+1) +1$ for all $n \geq 0$.
\end{theorem}

Note that we can have equality in Theorem \ref{thm:growthconj2} for infinitely many 
 values of $n$---see Example~\ref{ex:039982} above.
 
 Initially, we discovered this growth property empirically. 
 We prove it here by means of a graph construction related to the work of Flye Sainte-Marie. 
 We remark that another recent application of this idea leads to a bound on the maximal number of powers in finite words~\cite{BrlekLi-2025,LiPachockiRadoszewski-2024}. 
 
 Flye Sainte-Marie graphs (also known as Rauzy graphs or De Bruijn graphs) for sequences $\infw{x}$ can be defined in the following way: for an integer $ n\geq 0$, the graph $\Gamma_{\infw{x}}(n)$ is a directed graph whose vertex set is $\Fac_{\infw{x}}(n)$ and edge set is $\Fac_{\infw{x}}(n+1)$.
An edge $e \in \Fac_{\infw{x}}(n+1)$ starts at the vertex $u$ and ends at the vertex $v$ if and only if $u$ is a prefix of $e$ and $v$ is a suffix. To simplify the notation, for words $w_1, w_2$ of length $n$ and a word $w_3$ of length $n+1$, we write $w_1 \xrightarrow{w_3} w_2$ if $w_1$ and $w_2$ are a prefix and a suffix of $w_3$, respectively.

\begin{remark}
 Flye Sainte-Marie first introduced the previous graph construction in \cite{Flye-Sainte-Marie} to represent the overlaps between all the binary words of the same length. De Bruijn rediscovered the same graph construction in \cite{de-Bruijn46} and acknowledged Flye Sainte-Marie’s priority of \cite{Flye-Sainte-Marie} later in \cite{de-Bruijn75}. Rauzy used the same construction in \cite{Rauzy} to study the overlaps of factors in a word. 
\end{remark}

A directed graph is {\em (weakly) connected} if there is an undirected path between any pair of its vertices.

 The following proposition is “folklore”. To the authors' knowledge, the proof of this proposition does not seem to properly appear 
 in the literature. 

\begin{proposition}
 \label{prop:connected}
 Let $\infw{x}$ be a sequence.
 For all $n\ge 0$, the graph $\Gamma_{\infw{x}}(n)$ is connected.
\end{proposition}

\begin{proof}
Let $u$ and $v$ be two distinct length-$n$ factors of $\infw{x}$ and let us suppose that $u$ occurs before $v$ in $\infw{x}$. Then there are two integers $i<j$ such that $\infw{x}[i..i+n-1]=u$ and $\infw{x}[j.. j+n-1]=v$. Then the path $\infw{x}[i..i+n],\infw{x}[i+1..i+n+1],\ldots,\infw{x}[j-1..j+n-1]$ connects $u$ and $v$.
\end{proof}

\begin{remark}
 Proposition \ref{prop:connected} has already been (indirectly) mentioned in \cite[p.\ 199, lines 4--6]{Rote1994}, which states that every graph $\Gamma_{\infw{x}}(n)$ necessarily has a vertex $o$ from which every other vertex can be reached by a direct path. Thus each pair of vertices $u$, $v$ can be connected by an undirected path passing through $o$.
\end{remark}

 Let $\infw{x}$ be a sequence. For all $n \geq 0$, let $\refl_{\infw{x}}(n)$ be the set of all reflection-equivalent classes of length-$n$ factors of $\infw{x}$, i.e.,
\begin{equation*}
 \refl_{\infw{x}}(n)=\{\{u,u^R\} \suchthat u \in \Fac_{\infw{x}}(n)\},
\end{equation*}
 where $u$ may be equal to $u^R$. From the definition, the cardinality of the set $\refl_{\infw{x}}(n)$ is $r_{\infw{x}}(n)$.

Let $\Lambda_{\infw{x}}(n)$ be a directed graph such that its vertex set is $\refl_{\infw{x}}(n)$ and its edge set is a subset of $\Fac_{\infw{x}}(n+1)$.
An element $e \in \Fac_{\infw{x}}(n+1)$ is an edge of $\Lambda_{\infw{x}}(n)$ from $c_1$ to $c_2$ if and only if there exist two length-$n$ factors $u,v$ of $\infw{x}$ such that $u \in c_1$, $v \in c_2$, and $u \xrightarrow{e} v$.

\begin{remark}
From the definition, each element in $\Fac_{\infw{x}}(n+1)$ appears exactly once as an edge in $\Lambda_{\infw{x}}(n)$. This is easily seen by letting $e \in \Fac_{\infw{x}}(n+1)$ and by letting $u$ and $v$ be the length-$n$ prefix and suffix of $e$, respectively, and by using the property that $\{u, u^R\},\{v, v^R\} \in \refl_{\infw{x}}(n)$. Thus, the number of vertices in $\Lambda_{\infw{x}}(n)$ is $r_{\infw{x}}(n)$ and the number of edges in $\Lambda_{\infw{x}}(n)$ is $\rho_{\infw{x}}(n+1)$.
\end{remark}

\begin{remark}
Letting $\infw{x}$ be a sequence, the graph $\Lambda_{\infw{x}}(n)$ is connected for arbitrary $n \geq 0$. This follows in a direct way from $\Gamma_{\infw{x}}(n)$ being connected from Proposition \ref{prop:connected}.
 \end{remark}

\begin{lemma}
\label{lem:main}
Let $\infw{x}$ be a sequence.
Let $u,v$ be two distinct length-$(n+1)$ factors of $\infw{x}$ for some non-negative integer $n$. Then $u= v^R$ if and only if there are two distinct vertices $c_1, c_2$ in $\Lambda_{\infw{x}}(n)$ such that $u$ is an edge from $c_1$ to $c_2$ and $v$ is an edge from $c_2$ to $c_1$. 
\end{lemma}

\begin{remark}\label{remarkFig1}
If there are two distinct elements in $c_1$ and/or $c_2$, there may be at most two distinct edges from $c_1$ to $c_2$ and vice-versa. So, to prove Lemma~\ref{lem:main}, 
 we need to prove that for two distinct vertices, there is at most one edge from one to the other. Furthermore, if these two vertices are connected in both directions, then the two connecting edges in between are reflectively equivalent. 
 Graphically, Lemma \ref{lem:main} gives us that 
 $u= v^R$ and $u \neq v$ if and only if the graph $\Lambda_{\infw{x}}(n)$ has an occurrence of the pattern depicted in Figure~\ref{graphpat}.
\begin{figure}[htb]
\begin{center}
\begin{tikzpicture}[scale=0.2]
\tikzstyle{every node}+=[inner sep=0pt]
\draw [black] (27.5,-16) circle (3);
\draw (27.5,-16) node {$c_1$};
\draw [black] (49.6,-16) circle (3);
\draw (49.6,-16) node {$c_2$};
\draw [black] (29.667,-13.933) arc (127.71845:52.28155:14.519);
\fill [black] (29.67,-13.93) -- (30.61,-13.84) -- (29.99,-13.05);
\draw (38.55,-10.4) node [above] {$v$};
\draw [black] (47.349,-17.976) arc (-54.40916:-125.59084:15.119);
\fill [black] (47.35,-17.98) -- (46.41,-18.03) -- (46.99,-18.85);
\draw (38.55,-21.3) node [below] {$u$};
\end{tikzpicture}
\end{center}
\caption{The consequence of Lemma~\ref{lem:main} in the graph $\Lambda_{\infw{x}}(n)$, when $u = v^R$ but $u \not= v$.}
\label{graphpat}
\end{figure}
\end{remark}

\begin{proof}[Proof of Lemma~\ref{lem:main}] 
 If $u$ and $v = u^{R}$ are both edges of $\Lambda_{\text{{\bf x}}}(n)$, 
 then we now show that $u = (c_{1}, c_{2})$ and $v = (c_2, c_1)$ for some vertices $c_{1}$, $c_{2}$ of $\Lambda_{\text{{\bf x}}}(n)$ 
 (see Fig.\ \ref{graphpat}). 
 
 As indicated in Remark \ref{remarkFig1}, 
 we need to show that, 
 in the case where $c_1 \neq c_2$, there are at most two 
 edges between two distinct vertices, and if there are exactly two edges, then one is the reverse of the other. In this direction, the given 
 claim is clearly true for $n=0$ and $n=1$. 
 For $n \geq 2$, we now show that if
 there exists $u$ such that $u \neq u^R$ and $u, u^R \in \Fac_{\infw{x}}(n+1)$, then there exist two distinct vertices satisfying the conditions in the statement. Let $u=au'b$, where $a,b$ are letters. From the hypothesis that $u^R \in \Fac_{\infw{x}}(n+1)$, one has $au', u'b, u'^Ra, bu'^R \in \Fac_{\infw{x}}(n)$. Thus, $c_1=\{au', u'^Ra\}$ and $c_2=\{u'b, bu'^R\}$ are two vertices in the graph $\Lambda_{\infw{x}}(n)$. Moreover, since 
 $$au' \xrightarrow{au'b} u'b \text{ and }  bu'^R \xrightarrow{bu'^Ra} u'^Ra,$$ the word $au'b$ is an edge from $c_1$ to $c_2$ and $bu'^Ra$ is an edge from $c_2$ to $c_1$. 
 Finally, to prove that 
 $c_1 \neq c_2$, we can use
 a proof by contradiction
 together with a case analysis
 by considering 
 the $au'=bu'^R$ case 
 and the 
 $au'=u'b$ case. A similar approach
 can be used to prove that $u = v^{R}$.  We leave the details to the reader. 
\end{proof}

Now we introduce the Flye Sainte-Marie graphs for reflection-equivalent classes. Let $\infw{x}$ be a sequence. For all $n \geq 0$, let $K_{\infw{x}}(n)$ be a directed graph such that its vertex set is $\refl_{\infw{x}}(n)$ and its edge set is $\refl_{\infw{x}}(n+1)$.
Let $e \in \refl_{\infw{x}}(n+1)$. We define the edges of $K_{\infw{x}}(n)$ as follows:
\begin{itemize}
 \item if $|e \cap \Fac_{\infw{x}}(n+1)|=1$, there exist a unique word $u \in e$ and two vertices $c_1,c_2 \in \refl_{\infw{x}}(n)$ such that $u$ is from $c_1$ to $c_2$ in $\Lambda_{\infw{x}}(n)$. In the graph $K_{\infw{x}}(n)$, let $e$ be the edge from $c_1$ to $c_2$.
 \item if $|e \cap \Fac_{\infw{x}}(n+1)|=2$, from Lemma~\ref{lem:main} there exists a unique word $u \in e$ satisfying the following four conditions: 
 \begin{itemize}
 \item $u \neq u^R$;
 \item $u$ is lexicographically smaller than $u^R$; 
 \item $u, u^R \in \Fac_{\infw{x}}(n+1)$; 
 \item there exist two vertices $c_1,c_2 \in \refl_{\infw{x}}(n)$ such that $u$ is from $c_1$ to $c_2$ and $u^R$ is from $c_2$ to $c_1$ in $\Lambda_{\infw{x}}(n)$.
 \end{itemize} 
 In the graph $K_{\infw{x}}(n)$, let $e$ be the edge from $c_1$ to $c_2$.
\end{itemize}

\begin{proposition}
\label{prop: graph-final}
 Let $\infw{x}$ be a sequence. For all $n \geq 0$, the graph $K_{\infw{x}}(n)$ has exactly $r_{\infw{x}} (n)$ vertices and $r_{\infw{x}} (n+1)$ edges. Moreover, it is connected.
\end{proposition}

\begin{proof}
 From the construction, each element in $\refl_{\infw{x}}(n+1)$ appears exactly once as an edge in $K_{\infw{x}}(n)$ and the vertex set of this graph is exactly $\refl_{\infw{x}}(n)$. Thus, the number of vertices and edges is respectively $r_{\infw{x}} (n)$ and $r_{\infw{x}} (n+1)$.
 For the connectivity of the graph $K_{\infw{x}}(n)$, it is enough to show that every pair of vertices $c_1,c_2$ connected by a single edge $u$ in $\Lambda_{\infw{x}}(n)$ is also connected by the single edge $\{u,u^R\}$ in $K_{\infw{x}}(n)$. Since the graph $\Lambda_{\infw{x}}(n)$ is connected,
 the graph $K_{\infw{x}}(n)$ is also connected. 
\end{proof}

We can now complete the proof of Theorem~\ref{thm:growthconj2}:
\begin{proof}[Proof of Theorem~\ref{thm:growthconj2}]
This follows immediately from Proposition \ref{prop: graph-final}, 
since $K_{\infw{x}}(n)$ is a connected graph with $r_{\infw{x}} (n)$ vertices and $r_{\infw{x}} (n+1)$ edges. Thus
$r_{\infw{x}} (n) \leq r_{\infw{x}}(n+1) +1$.
\end{proof}

\section{Eventually periodic sequences}
\label{sectioneventual}
 
 We can characterize eventually periodic sequences
(i.e., sequences that are periodic from some index on)
in terms of their reflection complexity.

\begin{theorem}\label{thm:periodicity}
A sequence $\infw{x}$ is eventually periodic if
and only if both sequences $(r_{\infw{x}}(2n))_{n\ge 0}$ and 
$(r_{\infw{x}}(2n+1))_{n\ge 0}$ are eventually constant.
\end{theorem}

\begin{proof}
From Theorem~\ref{thm:paritygrowth} both sequences
$(r_{\infw{x}}(2n))_{n\ge 0}$ and $(r_{\infw{x}}(2n+1))_{n\ge 0}$ 
are non-decreasing. Also, from the inequalities 
$\frac{1}{2}\rho_{\infw{x}}(n) \leq r_{\infw{x}}(n)
\leq \rho_{\infw{x}}(n)$ in Theorem~\ref{thm:r-and-rho},
and the fact that the sequence $(\rho_{\infw{x}}(n))_{n\ge 0}$ is 
non-decreasing, we have that either the three integer 
sequences $(r_{\infw{x}}(2n))_{n\ge 0}$, 
$(r_{\infw{x}}(2n+1))_{n\ge 0}$, and $(\rho_{\infw{x}}(n))_{n\ge 0}$
are all bounded, or else none of them is.
Furthermore, we know that $(\rho_{\infw{x}}(n))_{n\ge 0}$ is
bounded if and only if the sequence $\infw{x}$ is 
eventually periodic (Theorem~\ref{thm:morse-hedlund}
above). Hence, we have two cases depending on the periodicity of $\infw{x}$.

\begin{itemize}
\item[(a)]
If $\infw{x}$ is eventually periodic, then 
$(\rho_{\infw{x}}(n))_{n\ge 0}$ is bounded, 
so $(r_{\infw{x}}(2n))_{n\ge 0}$ and 
$(r_{\infw{x}}(2n+1))_{n\ge 0}$ are eventually constant.
\item[(b)]
If $\infw{x}$ is not eventually periodic, its
factor complexity is not bounded, thus both sequences
$(r_{\infw{x}}(2n))_{n\ge 0}$ and 
$(r_{\infw{x}}(2n+1))_{n\ge 0}$ tend to infinity.
\end{itemize}
This ends the proof.
\end{proof}

\begin{remark}
If $\infw{x}$ is eventually periodic, the eventual 
values of $(r_{\infw{x}}(2n))_{n\ge 0}$ and 
$(r_{\infw{x}}(2n+1))_{n\ge 0}$ can be either equal or 
distinct, as seen from the examples of the sequences
$(01)^{\omega}$ and $(011)^{\omega}$.
\end{remark}

\section{Sturmian sequences and generalizations}\label{sec:Sturmian}
 In this section, we study Sturmian sequences as well as some generalizations. 
 First we state the following result, which 
characterizes Sturmian sequences in terms of their 
reflection complexity.

\begin{theorem}\label{thm:sturmian}
Let $\infw{x}$ be a non-eventually periodic sequence over a finite alphabet.
\begin{itemize}
\item[(a)] For all $n \geq 1$, we have $r_{\infw{x}}(n) \geq 1 + \lfloor \frac{n+1}{2} \rfloor$;
\item[(b)] We have $r_{\infw{x}}(n) = 1 + \lfloor \frac{n+1}{2} \rfloor$ if and only if $\infw{x}$ is Sturmian.
\end{itemize}
\end{theorem}

\begin{proof} For each integer $n \geq 1$, let ${\mathcal S}_n$ be the permutation group on $n$
elements. Let $\sigma_n$ be the permutation defined by
\[
\sigma_n :=
\begin{pmatrix}
1 &2 &\ldots &n-1 &n \\
n &n-1 &\ldots &2 &1 \\
\end{pmatrix}
\]
and $G_n$ be the subgroup of ${\mathcal S}_n$ generated by $\sigma_n$, i.e., the group
$\{ \sigma_n, {\rm id}_n\}$. The number of distinct orbits of $\{1, 2, \ldots, n \}$ under $G_n$ is 
equal to $n/2$ if $n$ is even, and to $(n+1)/2$ if $n$ is odd, which can be written
 $\lfloor (n+1)/2 \rfloor$ in both cases. Thus, applying~\cite[Theorem~1]{CharlierPuzyninaZamboni2017}
proves the first item of the theorem and the implication
$\Longrightarrow$ of the second item. 

To prove the last assertion, suppose that $\infw{x}$
is a Sturmian sequence. We know that every Sturmian 
sequence is reversal-closed (see
\cite[Theorem~4, p.\ 77]{DroubayPirillo1999}, where 
reversals are called mirror images). Furthermore, it is
proved in~\cite[Theorem~5, p.\ 77]{DroubayPirillo1999}
that a sequence is Sturmian if and only if it has one palindrome of all even lengths and two palindromes of all odd lengths.
Now, from Theorem~\ref{thm:r-and-rho}(b) we have
that
\[
r_{\infw{x}}(n) = \frac{1}{2} (\rho_{\infw{x}}(n) + 
\Pal_{\infw{x}}(n)) =
\begin{cases}
\frac{n+2}{2} = 1+ \lfloor \frac{n+1}{2} \rfloor, 
&\text{if $n$ even;} \\
\frac{n+3}{2} = 1+ \lfloor \frac{n+1}{2} \rfloor,
&\text{if $n$ odd}.
\end{cases}
\]
This completes the proof.
\end{proof}

 With regard to the above referenced work of 
Charlier et al.~\cite{CharlierPuzyninaZamboni2017}, 
also see the related and recent work by 
Luchinin and Puzynina~\cite{LuchininPuzynina2023}. 

\bigskip 

 The following is an analog of 
the Morse--Hedlund theorem (which is recalled in Theorem~\ref{thm:morse-hedlund} above).

\begin{corollary}\label{cor:periodic}
A sequence $\infw{x}$ is eventually periodic if and 
only if there exists $n \geq 1$ such that $r_{\infw{x}}(n) \leq \left\lfloor \frac{n+1}{2} \right\rfloor$.
Furthermore both sequences $(r_{\infw{x}}(2n))_n$
and $(r_{\infw{x}}(2n+1))_n$ are then eventually 
constant. 
\end{corollary}

\begin{proof}
Let $\infw{x}$ be a sequence. Contraposing Property~(a) 
of Theorem~\ref{thm:sturmian}, we obtain that if 
$r_{\infw{x}}(n) \leq \lfloor \frac{n+1}{2} \rfloor$ 
for some $n$, then $\infw{x}$ must be eventually 
periodic. 
Conversely, if $\infw{x}$ is eventually periodic,
it has a bounded number of factors, hence there exists
some integer $n$ for which the inequality of the 
statement is true. The last assertion is
Theorem~\ref{thm:periodicity} above.
\end{proof}

 Recall that every quasi-Sturmian sequence $\infw{x}$ 
 can be written as $\infw{x} = y f(\infw{z})$, where 
 $y$ is a word over a finite alphabet, $\infw{z}$ is a 
 (necessarily binary) Sturmian sequence, and $f$ an 
 {\em aperiodic morphism} from $\{0, 1\}$ to a finite 
 alphabet, see~\cite{Cassaigne1997,Coven1975,Paul197475}. 
 (Recall that $f$ aperiodic means that $f(01) \neq f(10)$.)
 We state the following theorem. 

\begin{theorem}\label{thm:quasi-sturmian}
 Let $\infw{x} = y f(\infw{z})$ be a quasi-Sturmian sequence, where $y$ is a word, $\infw{z}$ is a Sturmian sequence, and $f$ is an 
 aperiodic morphism  from $\{0, 1\}$ to a finite alphabet. Then 
\begin{itemize}
 \item[(a)]
 either $f(\infw{z})$ is reversal-closed and
 $r_{\infw{x}} (n)= \frac{n}{2} + O(1)$;
 \item[(b)]
 or else $f(\infw{z})$ is not reversal-closed and
 $r_{\infw{x}} (n) = n + O(1)$.
\end{itemize} 
\end{theorem}

\begin{proof}
 Since  $r_{\infw{x}}(n) = r_{f(\infw{z})}(n) + O(1)$,  it suffices to prove both statements for $r_{f(\infw{z})}$ instead of $r_{\infw{x}}$. 
 Since $\infw{z}$ and, hence, $f(\infw{z})$  are both uniformly recurrent, we can apply  Theorem~\ref{thm:uniformly-recurrent} to 
 $f(\infw{z})$.  The desired result then easily follows by 
 considering the case whereby 
 $f(\infw{z})$ is reversal-closed, so that
 $r_{f(\infw{z})}= \frac{1}{2}(\rho_{f(\infw{z})}
+ \Pal_{f(\infw{z})})$, 
 and the case such that it  
 is not reversal-closed, so that 
 $r_{f(\infw{z})}(n) = \rho_{f(\infw{z})}(n)$. 
\end{proof}

 Among several generalizations of Sturmian sequences, episturmian sequences have in particular the property---sometimes even taken 
 as part of their definition---to be reversal-closed.  Furthermore, their palindrome complexity has been studied~\cite{berstel2007,GlenJustin2009}. Here we develop a theorem similar
 to Theorem~\ref{thm:sturmian} above for these sequences.

\begin{definition}
Let $A$ be a finite alphabet with cardinality $\ell$.
A sequence $\infw{x}$ over $A$ is \emph{episturmian} if it is reversal-closed and has at most one left special factor of each length. 
An episturmian sequence $\infw{x}$ is \emph{$\ell$-strict} 
if it has exactly one left special factor of each length 
and every left special factor $u$ of $\infw{x}$ has $\ell$ 
distinct left extensions in $\infw{x}$.
\end{definition}

 We compute the reflection complexity of  episturmian sequences as follows. (Recall 
 that the factor complexity of an $\ell$-strict episturmian sequence is given 
 by $\rho_{\infw{x}}(n) = (\ell -1) n +1$.)

\begin{theorem}
Let $\infw{x}$ be an $\ell$-strict episturmian sequence.
Then, for all $n\ge 0$, we have
\[
r_{\infw{x}}(n) = 
(\ell-1) \left\lfloor \frac{n+1}{2} \right\rfloor + 1.
\]
\end{theorem}
\begin{proof}
 Let $\infw{x}$ be an $\ell$-strict episturmian sequence.
 The case $n=0$ is true.
 Assume that $n\ge 1$.
 Then by~\cite[Theorem~7]{DroubayJustinPirillo2001}, we have $\rho_{\infw{x}}(n) = (\ell-1)n+1$.
 We also know from \cite{DroubayPirillo1999} that
 \[
 \Pal_{\infw{x}}(n) 
 = 
 \begin{cases}
 1, & \text{if $n$ is even};\\
 \ell, & \text{if $n$ is odd}.
 \end{cases}
 \]
Using these facts, together with Theorem~\ref{thm:r-and-rho}(b), we deduce the desired result.
\end{proof}

\begin{example}
 For the Tribonacci sequence $\infw{tr}$, which is the fixed point of the morphism $0\mapsto 01, 1\mapsto 02, 2\mapsto 0$, we have $r_{\infw{tr}}(n) = 2 \left\lfloor \frac{n+1}{2} \right\rfloor + 1$ for all $n \geq 0$.
\end{example}

Since one interpretation of Sturmian sequences is
the binary coding of irrational trajectories on a 
square billiard table, we turn our attention toward 
 irrational trajectories
on a hypercube. Using results 
 due to Baryshnikov \cite{Baryshnikov1995}
 together with 
 Corollary~\ref{cor:growth}, 
 it follows in a direct way that 
 \[r_{\infw{x}}(n) \sim 
\frac{1}{2}
\sum_{k=0}^{\min{(s,n)}} k! 
 \binom{s}{k} \binom{n}{k}, 
\]
 when $n$ tends to infinity,  for an irrational billiard sequence 
 $\infw{x}$ on a  hypercube of dimension $(s+1)$.  In particular, we have that 
 $r_{\infw{x}}(n) = \Theta(n^s)$ when $n$ tends to  infinity. 

So-called complementation-symmetric Rote sequences, which were defined and studied 
 by Rote in 1994 
 \cite{Rote1994}, are related to Sturmian sequences as stated below in Theorem~\ref{thm:symbolization-Rote}.

\begin{definition}
Let $\infw{x}$ be a binary sequence.
Then $\infw{x}$ is called a \emph{Rote sequence} if its factor complexity satisfies $\rho_{\infw{x}}(n) = 2n$ for all $n\ge 1$.
The sequence $\infw{x}$ is said to be \emph{complementation-symmetric} if its set of factors is closed under the exchange morphism, i.e., if $w$ is a factor of $\infw{x}$, so is $E(w)$.
\end{definition}

We consider the mapping $\Delta\colon\{0,1\}^+\to \{0,1\}^*$ defined as follows: $\Delta(a)=a$ for all $a\in\{0,1\}$ and  $$\Delta(v(0)v(1)\cdots v(n))=u(0) u(1) \cdots u(n-1)$$ for $n\ge 1$ with $u(i) = (v(i+1)-v(i)) \bmod{2}$ for all $i\in\{0,\ldots,n-1\}$.
There is a natural extension of $\Delta$ to sequences: if $\infw{x}=(x(n))_{n\ge 0}$ is a binary sequence, then $\Delta(\infw{x})$ is the sequence whose $n$th letter is defined by $$(x(n+1)-x(n)) \bmod{2}$$ for all $n\ge 0$.
Observe that $\Delta(\infw{x})$ is the sequence of first differences of $\infw{x}$, taken modulo $2$.

\begin{theorem}[{\cite{Rote1994}}]
\label{thm:symbolization-Rote}
 A binary sequence $\infw{x}$ is a complementation-symmetric Rote sequence if and only if $\Delta(\infw{x})$ is Sturmian.
\end{theorem}

In fact, for each Sturmian sequence $\infw{s}$, there are two associated complem-entation-symmetric Rote sequences $\infw{x}$ and $\infw{x}'$ with $\infw{x}'=E(\infw{x})$.
The factors in $\infw{s}$ and its corresponding Rote sequences are closely related as shown below.

\begin{proposition}[{\cite{Rote1994}}; also see
{\cite[Proposition~2]{Medkova2018}} or 
{\cite[Lemma~2.7]{MedkovaPelantovaVuillon2019}}]
\label{prop:factors-Rote-seq}
 Let $\infw{s}$ be a Sturmian sequence and let $\infw{x}$ be the complementation-symmetric Rote sequence such that $\infw{s}=\Delta(\infw{x})$.
 Then $u$ is a factor of $\infw{s}$ if and only if both words $v,v'$ such that $u=\Delta(v)=\Delta(v')$ are factors of $\infw{x}$.
 Furthermore, for every $n\ge 0$, $u$ occurs at position $n$ in $\infw{s}$ if and only if $v$ or $v'$ occurs at position $n$ in $\infw{x}$.
\end{proposition}

\begin{lemma}
\label{lem:Rote-seq-stable-reversal}
 A complementation-symmetric Rote sequence is reversal-closed.
\end{lemma}
\begin{proof}
 Let $\infw{x}$ be a complementation-symmetric Rote sequence.
 Let $\infw{s}$ be the Sturmian sequence corresponding to $\infw{x}$, i.e., $\infw{s}=\Delta(\infw{x})$ given by Theorem~\ref{thm:symbolization-Rote}.
 Consider a factor $v$ of $\infw{x}$.
 Write $u=\Delta(v)$.
 Since $\infw{s}$ is reversal-closed, the word $u^R$ is also a factor of $\infw{s}$.
 Let $w$ and $w'$ be the binary words such that $u^R=\Delta(w)=\Delta(w')$ and $w'=E(w)$.
 By Proposition~\ref{prop:factors-Rote-seq}, both $w$ and $w'$ are factors of $\infw{x}$.
 Now observe that we have either $v^R=w$ or $v^R=w'$.
 This ends the proof.
\end{proof}

We compute the reflection complexity of Rote sequences as follows.

\begin{theorem}
 Let $\infw{x}$ be a complementation-symmetric Rote sequence.
 Then its reflection complexity satisfies $r_{\infw{x}}(n)=n+1$ for all $n\ge 0$.
\end{theorem}
\begin{proof}
Let $\infw{x}$ be a complementation-symmetric Rote sequence.
 We clearly have $r_{\infw{x}}(0)=1$.
 Now, for $n\ge 1$,~\cite[Theorem~8]{AlloucheBaakeCassaigneDamanik2003} states that $\Pal_{\infw{x}}(n)=2$.
 We finish the proof using Lemma~\ref{lem:Rote-seq-stable-reversal} and Theorem 
 \ref{thm:r-and-rho}(b).
\end{proof}

\section{Rich reversal-closed sequences}
\label{sectionrich}

Rich sequences have several equivalent definitions.
It is known that a word $w$ contains at most $|w|+1$ palindromic factors~\cite{DroubayJustinPirillo2001}.
A sequence is called \emph{rich} if each factor contains the maximal number of palindromic factors.

\begin{theorem}
\label{thm:characterization rich}
 Let $\infw{x}$ be a reversal-closed sequence.
 Then $\infw{x}$ is rich if and only if $r_{\infw{x}}(n+1)+r_{\infw{x}}(n) = \rho_{\infw{x}}(n+1)+1$ for all $n\ge 0$.
\end{theorem}
\begin{proof}
From~\cite[Theorem~1.1]{BucciDeLucaGlenZamboni2009}, the sequence $\infw{x}$ is rich if and only if the inequality in Theorem~\ref{thm:Pal-first-diff-rho}(b) is an equality.
The result then follows from Theorem~\ref{thm:r-and-rho}(b).
\end{proof}

 Those among binary quasi-Sturmian sequences 
that are codings of rotations are rich, 
see~\cite[Theorem~19]{BMBLV}.

\begin{corollary}\label{cor:quasi-and-rich}
 Let $\infw{x}$ be a binary reversal-closed quasi-Sturmian sequence.
 There exists a constant $C$ such that 
 $r_{\infw{x}}(n+1)+r_{\infw{x}}(n) = n + C$ 
 for $n$ large enough.
\end{corollary}
\begin{proof}
Let $C'$ be a constant such that $\rho_{\infw{x}}(n) = n + C'$ for $n$ large enough.
It is enough to choose $C=C'+2$.
\end{proof}

Using Theorem~\ref{thm:r-and-rho}(b) and~\cite[Corollaries~2.27 and~2.29]{Rukavicka2021}, it is possible to bound the reflection complexity of rich sequences as follows.

\begin{proposition}
 Let $\infw{x}$ be a rich sequence over an alphabet of $q$ letters and write $\delta = \frac{2}{3(\log 3- \log 2)}$.
 Then $r_{\infw{x}}(n) \le \frac{nq}{2} (2q^2n)^{\delta \log n} (1 + nq^3 (2q^2n)^{\delta \log n})$ for all $n\ge 1$.
\end{proposition}

Other sequences have a reflection complexity satisfying the equality of Theorem~\ref{thm:characterization rich}.
For instance, it is the case for complementation-symmetric sequences, sequences canonically associated with some specific Parry numbers, and sequences coding particular interval exchange transformations.
For more details, see~\cite[Section~3]{BalaviMasakovaPelantova2007}.

\section{Automatic sequences}\label{sec:automatic}
 In this section, we study the reflection complexity of automatic sequences.
First, in a positional numeration system $U$ having an \emph{adder} (i.e., addition is recognizable by an automaton reading $U$-representations), we show that if a sequence is $U$-automatic, then its reflection complexity is a $U$-regular sequence.
Furthermore we show how to effectively compute a linear representation for the sequence, making use of the free software \texttt{Walnut}~\cite{Mousavi2016,Shallit2023Walnut}.
 Next, we explore the reflection complexity of some famous automatic sequences, namely the Thue--Morse, the period-doubling, generalized paperfolding, generalized Golay--Shapiro, and the Baum--Sweet sequences.

\subsection{Reflection complexity is computably regular}
 We now show that the reflection complexity of an automatic sequence is regular.

\begin{theorem}
Let $U=(U(n))_{n\ge 0}$ be a positional numeration system such that there is an adder, and let $\infw{x}$ be a $U$-automatic sequence. 
Then $(r_{\infw{x}}(n))_{n\ge 0}$ is a $U$-regular sequence.
Furthermore, a linear representation for 
$(r_{\infw{x}}(n))_{n\ge 0}$ is computable from 
the DFAO for $\infw{x}$.
\end{theorem}

\begin{proof}
Here is a sketch of the proof before we give the details: We
create a first-order logical formula asserting that the factor
$\infw{x}[i..i+n-1]$ is the first occurrence of this factor, or its reversal. Then the number of such $i$ is precisely the reflection complexity at $n$.
From this, we can create
a linear representation for the number of such $i$. 

Now some more details. We define the following logical formulas:
\begin{align}
\FactorEq(i,j,n) &:= \forall t \ (t<n) \implies \infw{x}[i+t]=\infw{x}[j+t] \nonumber\\
\FactorRevEq(i,j,n) &:= \forall t \ (t<n) \implies \infw{x}[i+t]=\infw{x}[(j+n)-(t+1)] \label{eq:three-formulas}\\
\RefComp(i,n) &:= \forall j \ (j<i) \implies ((\neg\FactorEq(i,j,n)) & \nonumber \\
& \andd (\neg\FactorRevEq(i,j,n))). \nonumber 
\end{align}

Now we use the fundamental result on B\"uchi arithmetic to translate each of these formulas to their corresponding automata accepting the base-$U$ representation of those pairs $(i,n)$ making the formula true. Next, we use a basic result to convert
the automaton for $\RefComp$ to the corresponding linear representation computing the reflection complexity. 
\end{proof}

Once we have a linear representation for the reflection complexity, we can easily compute it for 
 a given $n$. Furthermore, we can compare it to a guessed formula, provided that this 
 formula can also be expressed as a linear representation (see~\cite{Shallit2023Walnut}). In the next section we carry this out in detail for a number of famous sequences. 

   First we study the case of the Thue-Morse sequence. We want    to emphasize the fact that the point of this example is   
   {\em not} to reprove things that could be done more easily 
by appealing to existing theorems. It is to illustrate how 
the approach via \texttt{Walnut} can, in principle, carry 
out the various constructions for an arbitrary word, using 
Thue-Morse as an example simple enough where the matrices 
can actually be displayed. 

\medskip

We can compute a linear representation for the reflection complexity $r_{\infw{t}} (n)$ of the $2$-automatic Thue--Morse sequence $\infw{t}$, using the
 same approach as in the preceding section. Here we use the following \texttt{Walnut} code:
\begin{verbatim}
def factoreq_tm "At (t<n) => T[i+t]=T[j+t]"::
def factorreveq_tm "At (t<n) => T[i+t]=T[(j+n)-(t+1)]"::
def rc_tm n "Aj (j<i) => ((~$factoreq_tm(i,j,n)) 
 & (~$factorreveq_tm(i,j,n)))"::
\end{verbatim}
This generates a linear representation of rank $66$, which can be minimized to the following.
$$
 v = \left[\begin{array}{ccccccccc}
1&0&0&0&0&0&0&0&0
\end{array}
\right], \quad 
 w = \left[\begin{array}{ccccccccc}
1& 2& 3& 4& 6& 6&10&10&13
\end{array}
\right]^T
$$
\begin{equation}
\mu(0) = \frac{1}{33} \left[\begin{smallarray}{rrrrrrrrr}
33& 0& 0& 0& 0& 0& 0& 0& 0\\
 0& 0& 33& 0& 0& 0& 0& 0& 0\\
 0& 0& 0& 0& 33& 0& 0& 0& 0\\
 0& 0& 0& 0& 0& 0& 33& 0& 0\\
 0& 0& 0& 0& 0& 0& 0& 0& 33\\
 0& 0&-26& 0& 0& 23& 10&-10& 36\\
 0& 0&-57& 33& 0& 6& -6& 6& 51\\
 0& 0&-79& 33& 33& -5&-28& 28& 51\\
 0& 0&-72& 0& 33& 18&-18& 18& 54
 \end{smallarray}\right], \quad 
 \mu(1) = \frac{1}{33}\left[\begin{smallarray}{rrrrrrrrr}
0& 33& 0& 0& 0& 0& 0& 0& 0\\
 0& 0& 0& 33& 0& 0& 0& 0& 0\\
 0& 0& 0& 0& 0& 33& 0& 0& 0\\
 0& 0& 0& 0& 0& 0& 0& 33& 0\\
 0& 0&-24& 0& 0& 39& -6& 6& 18\\
 0& 0&-40& 0& 0& 43&-10& 10& 30\\
 0& 0&-78& 33& 33& 3&-36& 36& 42\\
 0& 0&-86& 33& 33& 5&-38& 38& 48\\
 0& 0&-72& 0& 0& 51&-18& 18& 54
\end{smallarray}\right].
 \label{eq:tmrep}
\end{equation}

 Recall that Brlek~\cite{Brlek1989}, 
 de Luca and Varricchio~\cite{deLucaVarricchio1989},
 and Avgustinovich~\cite{Avgustinovich1994} 
 independently
gave a simple recurrence for the number of length-$n$ factors of $\infw{t}$, namely $ \rho_{\infw{t}}(2n) = \rho_{\infw{t}}(n) + \rho_{\infw{t}}(n+1)$ and $\rho_{\infw{t}}(2n+1) = 2 \rho_{\infw{t}}(n+1)$ for $n\ge 2$. As it turns out, there is a simple relationship between $r_{\infw{t}}$ and $\rho_{\infw{t}}$.

\begin{theorem}\label{thm:theoremrTM}
Let $\infw{t}$ be the Thue--Morse sequence.
\begin{itemize}
 \item[(a)]
For all $n\ge 0$, we have $r_{\infw{t}}(2n+1) = \rho_{\infw{t}}(n + 1)$.
 \item[(b)]
 For all $n \geq 2$, we have 
\[ r_{\infw{t}}(2 n) = \begin{cases} 
 \rho_{\infw{t}}(n+1) + 1, & \text{if $\exists m\ge 0$ with $3 \cdot 4^{m-1} + 1 \leq n \leq 4^{m}$}; \\ 
 \rho_{\infw{t}}(n+1), & \text{otherwise}. 
 \end{cases} 
\]
\item[(c)] There is an automaton of $14$ states that 
computes the first difference 
$(r_{\infw{t}}(n+1)-r_{\infw{t}}(n))_{n\ge 0}$.
\end{itemize}
\end{theorem}

\begin{proof}
We prove each item separately.
\begin{itemize}
 \item[(a)]
 Above in Equalities~\eqref{eq:tmrep} we computed a linear representation
 for $r_{\infw{t}}(n)$. From this linear
 representation we can easily compute
 one for $r_{\infw{t}}(2n+1)$ merely by replacing
 $w$ with $\mu(1) w$. (Indeed, base-$2$ representations of integers $2n+1$ all end with $1$.)

 Next, we can compute a linear representation for $\rho_{\infw{t}} (n+1)$ using the following \texttt{Walnut} command.
 \begin{verbatim}
def sc_tm_offset n "Aj (j<i) => ~$factoreq_tm(i,j,n+1)":
\end{verbatim}
 This creates a linear representation of rank $6$. 

 Finally, we use a block matrix construction to compute a linear representation for the difference $r_{\infw{t}}(2n+1) -\rho_{\infw{t}}(n + 1)$ and minimize it; the result is the $0$ representation. This computation gives a rigorous proof of item (a).
 
 \item[(b)] This identity can be proven in a similar way. We form the linear representation for $$r_{\infw{t}} (2n) - \rho_{\infw{t}} (n+1) - [\exists m \colon 3 \cdot 4^{m-1} + 1 \leq n \leq 4^{m}],$$
 where the last term uses the Iverson bracket.
 We then minimize the result and obtain the
 $0$ representation.

 \item[(c)] We can compute a linear representation for the first difference
 $(r_{\infw{t}}(n+1)-r_{\infw{t}}(n))_{n\ge 0}$, and then
 use the ``semigroup trick'' \cite[Section~4.11]{Shallit2023Walnut} to prove that the difference is bounded and find the automaton for it. It is displayed in Figure~\ref{aut7}.
 \begin{figure}[H]
 \begin{center}
 \includegraphics[width=5.2in]{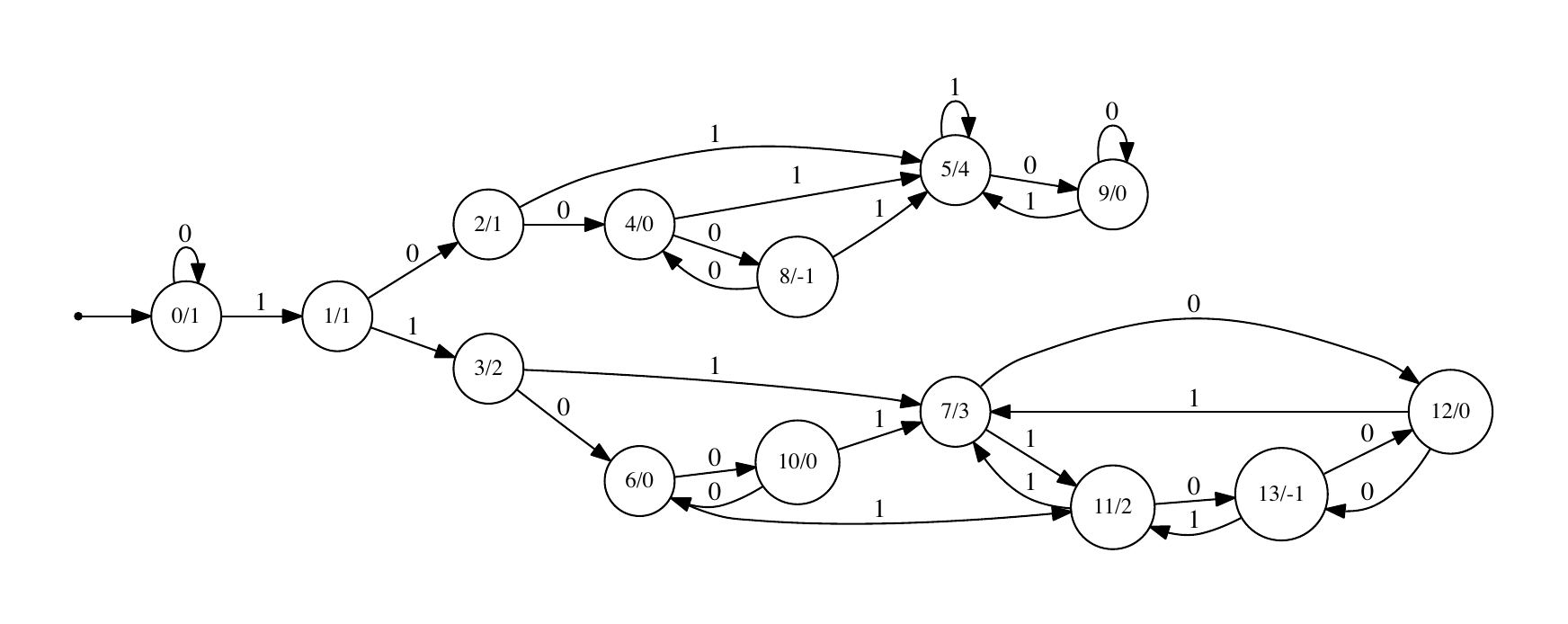}
 \end{center}
 \caption{Automaton computing $(r_{\infw{t}}(n+1)-r_{\infw{t}}(n))_{n\ge 0}$ where $\infw{t}$ is the Thue--Morse sequence.}
 \label{aut7}
 \end{figure}
 \end{itemize}
 These computations rigorously prove the three items of the claim.
\end{proof}
 
 A similar approach can be used to evaluate the reflection complexity for the \emph{period-doubling sequence} $\infw{p}$, 
 which may be defined 
 as the fixed point 
 of the morphism $0 \mapsto 01$ and $1 \mapsto 00$. 
 This gives us that $\infw{p}$ is $2$-automatic as well.
 By analogy with 
 Theorem~\ref{thm:theoremrTM}, we can 
 show that 
 for all $n\ge 0$, we have $r_{\infw{p}}(2n+1) = \rho_{\infw{p}}(n) + 1$, 
and, for all $n \geq 2$, we have 
\[ 
 r_{\infw{p}}(2 n) = 
 \begin{cases} 
 \rho_{\infw{p}}(n+1) - 1, &\text{if $\exists m
 \ge 0$ with $3 \cdot 2^{m-1} \leq n \leq 2^{m+1} - 1$}; \\ 
 \rho_{\infw{p}}(n+1) - 2, &\text{otherwise}, 
 \end{cases} 
\]
 and we may similarly devise an analogue of part (c) 
 of Theorem \ref{thm:theoremrTM}. 
 
 A \emph{paperfolding sequence} $\infw{p}_{\infw{f}}$ is a 
 binary sequence $p_1 p_2 p_3 \cdots$
specified by a sequence of
\emph{binary unfolding instructions} $f_0 f_1 f_2 \cdots$,
as the limit of the sequences 
$\infw{p}_{f_0 f_1 f_2 \cdots}$, defined as follows:
\[
 \infw{p}_{\eps} = \eps
 \quad \text{and} \quad
 \infw{p}_{f_0 \cdots f_{i+1}} = \infw{p}_{f_0 \cdots f_i} \ f_{i+1} \ E(\infw{p}_{f_0 \cdots f_i}^R) \; \text{ for all }i\ge 0
\]
where $E$ is the exchange morphism.
For example, if $\infw{f} = 000\cdots$, we get
the simplest paperfolding sequence
$$\infw{p} = 0010011000110110001001110011011 \cdots.$$
Note that a paperfolding sequence is $2$-automatc if and only if the sequence of unfolding instructions is eventually periodic~\cite[Theorem~6.5.4]{AlloucheShallit2003}.

Allouche~\cite{Allouche1997}, and later, Baake~\cite{Baake1999} proved that no paperfolding sequence
contains a palindrome of length $>13$. In fact, even more is true as shown below.
\begin{proposition}
No paperfolding sequence contains a reflected factor of length $>13$.
\label{prop:pf13}
\end{proposition}

\begin{proof}
It suffices to show that no paperfolding sequence contains a reflected factor of length $14$. For if this holds, but there is a longer reflected factor $x$,
we could write $x = yz$ where $|y| = 14$. Then
$x^R = z^R y^R$, so $y$ would be a reflected factor of 
length $14$, a contradiction.

Now, by a known result on the appearance function of paperfolding sequences~\cite[Theorem~12.2.1]{Shallit2023Walnut}, we know that every length-$14$ factor of a paperfolding sequence $\infw{p}_{\infw{f}}$ appears in a prefix of length $109$, which is in turn specified by the first 7 unfolding instructions. We can then simply examine each of the $56$ length-$14$ factors of these $128$ (finite) words and verify that no factor is reflected. 
\end{proof}
 
 A \emph{generalized Golay--Shapiro sequence} $\infw{g}$ is defined by
taking the running sum, modulo $2$, of a paperfolding sequence $\infw{p}_{\infw{f}}$. The famous \emph{Golay--Shapiro sequence } (also called the \emph{Rudin--Shapiro sequence})~\cite{Golay:1949,Golay:1951,Rudin:1959,Shapiro:1952}
corresponds to the case of unfolding instructions
$0(01)^\omega$~\cite[Definition~6]{AlloucheBaakeCassaigneDamanik2003}.
Note that a generalized Golay--Shapiro sequence is
$2$-automatic if and only if its corresponding generalized paperfolding sequence is $2$-automatic.

 The following analogue of 
 Proposition~\ref{prop:pf13}
 can be proved like Proposition~\ref{prop:pf13}, 
 by appealing to a known result on the recurrence 
 function of generalized Golay--Shapiro sequences 
\cite[Proposition~4.1]{Allouche&Bousquet-Melou:1994b},
 giving that 
 every length-$15$ factor of a paperfolding sequence $\infw{p}_{\infw{f}}$ appears in a prefix of length $2408$, which is in turn specified by the first 12 unfolding instructions. 
 
\begin{proposition} \label{prop:gs15}
 No generalized Golay--Shapiro sequence contains a reflected factor of length $>14$.
\end{proposition}

 We can now prove the following result.

\begin{theorem}
Let $\infw{g}$ be a generalized Golay--Shapiro sequence.
\begin{itemize}
 \item[(a)] For all $n \geq 15$, we have $r_{\infw{g}}(n) = 
 \rho_{\infw{g}} (n) = 8n-8$.
 \item[(b)] The reflection complexity of every generalized Golay--Shapiro sequence is the same, and takes the values
 $2$, $3$, $6$, $10$, $14$, $22$, $30$, $42$, $48$, $62$, $72$, $83$, $92$, $103$
 for $n \in [1,14]$.
\end{itemize}
\end{theorem}

\begin{proof}
We prove each item separately.
\begin{itemize}
 \item[(a)]
For $n \geq 15$, the result follows from 
combining the results of
Allouche and Bousquet-Melou~\cite{Allouche&Bousquet-Melou:1994b} and
Proposition~\ref{prop:gs15}.

\item[(b)] The result for $n \geq 15$ follows from Item (a). For $n < 15$ the result can be verified by enumeration of all length-$2408$ prefixes of paperfolding sequences specified by instructions of length $12$.
\end{itemize}
This ends the proof.
\end{proof} 

Let the \emph{Baum--Sweet sequence}
 $$ \infw{b} = (b(n))_{n\ge 0} = 1101100101001001100100000100100101001001\cdots $$
 be defined by $b(0) = 1$ and for $n\ge 1$, $b(n)$ 
 is $1$ if the base-2 expansion of $n$ contains 
 no block of successive zeros of odd length 
 and $0$ otherwise. 
 It is $2$-automatic as well.
 The factor complexity function for 
 $ \infw{b} $ starts with
\begin{equation}\label{eq:BaumSweetrho} 
 (\rho_{\infw{b}}(n))_{n\ge 0} 
 = 1, 2, 4, 7, 13, 17, 21, 27, 33, 38, 45, 52, 59, 65, 70, \ldots 
\end{equation}
 and the reflection complexity function for $ \infw{b} $ 
 with
\begin{equation}\label{eq:BaumSweetr}
 (r_{\infw{b}}(n) )_{n\ge 0} 
 = 1, 2, 3, 5, 8, 11, 13, 17, 21, 25, 30, 35, 40, 46, 50, 56, \ldots. 
\end{equation}

We can again compute a linear representation for $(r_{\infw{b}} (n))_{n\ge 0}$ using the following \texttt{Walnut} code:
\begin{verbatim}
def factoreq_bs "At (t<n) => BS[i+t]=BS[j+t]"::
def factorreveq_bs "At (t<n) => BS[i+t]=BS[(j+n)-(t+1)]"::
def rc_bs n "Aj (j<i) => ((~$factoreq_bs(i,j,n)) 
 & (~$factorreveq_bs(i,j,n)))"::
\end{verbatim}
This gives us a linear representation of rank $90$. From this linear representation, a computation proves the following result.

\begin{corollary}\label{seconddiffBaumSweet}
Let $\infw{b}$ be the Baum--Sweet sequence.
Then the first difference of the sequence $r_{\infw{b}} (n)$ is $2$-automatic, over the alphabet $\{ 1,2, \ldots, 8 \}$.
\end{corollary}

\section{Further directions}\label{sec:conclusion}
 We conclude the paper by considering some further research directions to pursue in relation to 
 reflection complexities of sequences and by raising some open problems.

 We encourage further explorations of the evaluation 
of $r_{\infw{x}}$ for sequences $\infw{x}$ for which 
properties of $\Pal_{\infw{x}}$ and/or $\rho_{\infw{x}}$ 
are known, especially if \texttt{Walnut} cannot be used 
directly in the investigation of $r_{\infw{x}}$. 
 For example, by letting the \emph{Chacon sequence} 
$\infw{c}$ be the fixed point of the morphism 
$0 \mapsto 0010$ and $1 \mapsto 1$, it is known that 
$\Pal_{\infw{c}}(n)=0$ for all $n\ge 13$. Also, its 
factor complexity satisfies $\rho_{\infw{c}}(n)=2n-1$ 
for $n\ge 2$~\cite{Ferenczi1995}.
We have
\[
(r_{\infw{c}}(n))_{n\ge 0} = 1, 2, 2, 4, 4, 6, 7, 10, 11, 14, 16, 20, 23, 25, 27, 29, 31, 33, \ldots.
\]
The sequence $\infw{c}$ is not automatic in a given so-called 
\emph{addable} numeration system (i.e., where there is an adder). 
Therefore, we cannot use \texttt{Walnut}, in this case. 
 However, an inductive argument can be applied to prove  that 
 $r_{\infw{c}}(n)=\rho_{\infw{c}}(n)$ for all $n\ge 13$. 

 We propose the following natural questions.

 \begin{question}
 To what extent can the reflection complexity
 be used to discriminate between different families of sequences, by analogy with our characterizations of Sturmian and eventually periodic sequences? 
 \end{question}

 The complexity function $\Unr_{\infw{x}}$ defined above may be of interest in its own right, as is the case with the ``reflection-free'' 
 complexity function enumerating factors such that the reversal of every sufficiently large factor is not a factor. 

\begin{question}
How can Theorem~\ref{thm:theoremrTM} be generalized 
with the use of standard generalizations of the 
Thue--Morse sequence? 
 \end{question}

 For example, if we let 
 $$ \infw{t3} = (t3(n))_{n\ge 0} = 011212201220200112202001200\cdots $$
 denote the generalized Thue--Morse sequence for which the $n$th term $t3(n)$ is equal to the number of $1$'s, modulo 3, in the base-$2$ 
 expansion of $n$, it can be shown that 
 $ r_{\infw{t3}}(n) = \rho_{\infw{t3}}(n) $
 for all $n \geq 3$, 
 and it appears that a similar property
 holds for the cases given by taking the number of $1$'s modulo $\ell > 4$. 

\begin{question}
 What is the reflection complexity of the Thue--Morse sequence 
 over polynomial extractions, with regard to the work of Moshe~\cite{Moshe2007}?
\end{question}

\begin{question}
 How can the upper bound in Theorem~\ref{thm:upper-bound-abcd}
 be improved? If $r_{\infw{x}}(n)$ is of the 
 form $\Omega(n)$, then how can this be improved?
\end{question}

\begin{question}
 How does the reflection complexity compare with other complexity functions, e.g., the complexity functions 
 listed in Section~\ref{sec:intro}?
\end{question}

 This leads us to ask about the respective growths of the complexity functions listed in Section \ref{sec:intro}, in particular 
 for morphic sequences. In this direction, recall that the factor complexity of a pure morphic sequence is either $\Theta(1)$, $\Theta(n)$, $\Theta(n \log \log n)$, $\Theta(n \log n)$ or $\Theta(n^2)$, see \cite{Pansiot1984} 
(more details can be found, e.g., in~\cite{Cisternino2018}). 
The factor complexity of a morphic sequence is either
$\Theta(n^{1+1/k})$ for some positive integer $k$, or else
it is $O(n \log n)$, see \cite{Devyatov2018,Deviatov2008}). 
As an illustration with a result that has not 
been already cited above, a comparison between growths for 
the factor complexity and the Lempel-Ziv complexity can be 
found in \cite{ConstantinescuIlie2007}. We end with an easy 
result for the growth of the reflection complexity in the 
case of pure morphic and morphic sequences. 

\begin{proposition}

\ { }

\begin{itemize}
\item[$*$]
The reflection complexity of a pure morphic sequence 
is either $\Theta(1)$, $\Theta(n)$, 
$\Theta(n \log \log n)$, $\Theta(n \log n)$ or 
$\Theta(n^2)$.
\item[$*$]
The reflection complexity of a morphic sequence is either
$\Theta(n^{1+1/k})$ for some positive integer $k$, or else
it is $O(n \log n)$.
\end{itemize}
\end{proposition}

\begin{proof} 
Use the inequalities in Theorem~\ref{thm:r-and-rho}:
for a sequence $\infw{x}$ and for all $n\ge 0$, we have $\frac{1}{2}\rho_{\infw{x}}(n) 
\leq r_{\infw{x}}(n)
\leq \rho_{\infw{x}}(n)$.
\end{proof}

\subsection*{Acknowledgments}
 We thank Boris Adamczewski for interesting discussions. John Campbell is grateful to acknowledge support from a Killam Postdoctoral 
 Fellowship from the Killam Trusts, and thanks Karl Dilcher for a useful discussion. The research of Jeffrey Shallit is supported by NSERC 
 grant 2024-03725. Manon Stipulanti is an FNRS Research Associate supported by the Research grant 1.C.104.24F. The authors are 
 thankful for a number of detailed reviewer reports that have substantially improved this paper.

 \

Jean-Paul Allouche

CNRS, IMJ-PRG 

 Sorbonne, 4 Place Jussieu

 75252 Paris Cedex 05, France

{\tt jean-paul.allouche@imj-prg.fr}

 \ 

 John M.\ Campbell 

 Department of Mathematics and Statistics, Dalhousie University

6299 South St., Halifax, NS B3H 4R2, Canada 

{\tt jmaxwellcampbell@gmail.com}

 \ 

 Shuo Li

 Department of Mathematics and Statistics, University of Winnipeg
 
515 Portage Avenue, Winnipeg, MB R3B 2E9, Canada

{\tt shuo.li.ismin@gmail.com}

 \ 

Jeffrey Shallit 

School of Computer Science, University of Waterloo

Waterloo, ON N2L 3G1, Canada

{\tt shallit@uwaterloo.ca}

 \ 

Manon Stipulanti

 Department of Mathematics, University of Li\`{e}ge

 4000 Li\`{e}ge, All\'{e}e de la D\'{e}couverte 12, Belgium 

{\tt m.stipulanti@uliege.be}

\end{document}